\documentclass[a4paper,11pt]{article}
\usepackage{amsmath,amsthm,amssymb}
\usepackage[mathscr]{eucal}

\usepackage[bookmarks=false]{hyperref}
\usepackage{color,soul}

{\theoremstyle{definition}\newtheorem{definition}{Definition}
\newtheorem{notation}[definition]{Notation}
\newtheorem{remark}[definition]{Remark}}
\newtheorem{proposition}[definition]{Proposition}
\newtheorem{lemma}[definition]{Lemma}
\newtheorem{theorem}[definition]{Theorem}

\def\SLtwo{\mathrm{SL}_2(\mathbb{Z})}
\def\SLn{\mathrm{SL}_n(\mathbb{Z})}
\def\SLnR{\mathrm{SL}_n(\mathbb{R})}
\def\compZ{\mathcal{Z}}
\newcommand{\C}{\mathbb{C}}
\newcommand{\embed}[1]{\prec_{#1}}
\newcommand{\cR}{\mathcal{R}}
\newcommand{\acts}{\curvearrowright}
\newcommand{\actson}{\curvearrowright}
\newcommand{\SL}{\operatorname{SL}}
\newcommand{\rL}{\mathord{\text{\rm L}}}
\newcommand{\mcL}{\mathord{\mathcal L}}
\newcommand{\N}{\mathbb{N}}
\newcommand{\T}{\mathbb{T}}

\newcommand{\cF}{\mathcal{F}}
\newcommand{\id}{\mathord{\operatorname{id}}}
\newcommand{\si}{\sigma}
\newcommand{\recht}{\rightarrow}
\newcommand{\cU}{\mathcal{U}}
\newcommand{\vphi}{\varphi}
\newcommand{\cW}{\mathcal{W}}

\newcommand{\al}{\alpha}
\newcommand{\eps}{\varepsilon}
\newcommand{\Tr}{\operatorname{Tr}}
\newcommand{\ovt}{\mathbin{\overline{\otimes}}}
\newcommand{\B}{\operatorname{B}}
\newcommand{\om}{\omega}
\newcommand{\ot}{\otimes}
\newcommand{\Ad}{\operatorname{Ad}}
\newcommand{\cN}{\mathcal{N}}

\begin{document}

\begin{center}
{\boldmath\LARGE\bf II$_1$ factors and equivalence relations with distinct fundamental groups}

\bigskip

{\sc by Jan Keersmaekers\footnote{K.U.Leuven, Department of Mathematics, jan.keersmaekers@wis.kuleuven.be \\ Supported by KU Leuven BOF research grant OT/08/032} and An Speelman\footnote{K.U.Leuven, Department of Mathematics, an.speelman@wis.kuleuven.be \\ Research Assistant of the Research Foundation --
    Flanders (FWO)}}
\end{center}

\begin{abstract}\noindent
We construct a group measure space II$_1$ factor that has two non-conjugate Cartan subalgebras. We show that the fundamental group of the II$_1$ factor is trivial, while the fundamental group of the equivalence relation associated with the second Cartan subalgebra is non-trivial. This is not absurd as the second Cartan inclusion is twisted by a 2-cocycle.
\end{abstract}

\section{Introduction and statement of the main results}

Murray and von Neumann introduced the fundamental group for type II$_1$ factors in \cite{MvN43}. This notion is not related to the fundamental group of a topological space, but instead is a subgroup of the positive real numbers. Murray and von Neumann showed that the fundamental group of the hyperfinite type II$_1$ factor is $\mathbb{R}^*_+$ itself. They were not able to compute the invariant for any other II$_1$ factor. Their line ``the general behavior of the above invariants - including fundamental groups - remains an open question'' still has some truth today. Calculating the fundamental group of a II$_1$ factor is a very difficult problem.

In fact, it took almost 40 years before one proved that the fundamental group of a II$_1$ factor can be different from $\mathbb{R}^*_+$. The first examples of II$_1$ factors with ``small'' fundamental group were given in \cite{C80}, as Connes showed that the fundamental group of a property (T) factor is countable. However up to now, one is unable to compute the fundamental group of any property (T) factor.

The first explicit computations of fundamental groups of II$_1$ factors, other than the hyperfinite one, were established around 2002. Using his deformation-rigidity theory, Popa proved that the fundamental group of the group von Neumann algebra $\mathcal{L}(\SLtwo \ltimes \mathbb{Z}^2)$ is trivial (\cite{Po02}). Later he showed that any countable subgroup of $\mathbb{R}^*_+$ arises as the fundamental group of a type II$_1$ factor (\cite{Po03}). Alternative constructions for these results can be found in \cite{IPP05} and \cite{H07}.

In \cite{PV08a} and \cite{PV08c}, Popa and Vaes showed that a large class of uncountable subgroups of $\mathbb{R}^*_+$ appear as the fundamental group of a type II$_1$ factor. This was done through an existence theorem, using a Baire category argument. In \cite{D10}, Deprez provides explicit examples of this phenomenon.

In this paper we study the link between the fundamental group of a II$_1$ factor and the fundamental group of the equivalence relation associated to a Cartan inclusion in the factor. The notion of fundamental group for countable equivalence relations is the natural translation of the notion for II$_1$ factors. More or less parallel to the results mentioned above, rigidity theory for equivalence relations became more developed, yielding calculations of fundamental groups of equivalence relations in \cite{GG87}, \cite{Ga01}.

In \cite{FM75}, it was shown that an inclusion of a Cartan subalgebra in a II$_1$ factor gives rise to a II$_1$ equivalence relation $\mathcal{R}$ and a scalar 2-cocycle. If the factor is of group measure space type, i.e. $M = \rL^\infty(X) \rtimes G$ where $G \actson (X,\mu)$ is a free ergodic probability measure preserving action, then $\mathcal{R}$ is the orbit equivalence relation of the action $G \acts X$ and the 2-cocycle is trivial. It follows immediately from the definitions that the fundamental group of the orbit equivalence relation $\mathcal{R}(G \curvearrowright X)$ is a subgroup of $\mathcal{F}(\mathrm{L}^\infty(X)\rtimes G)$. In many cases, one calculates the fundamental group of a group measure space II$_1$ factor by proving that equality holds. However, in \cite[\S 6.1]{Po06} Popa gives an example of a free ergodic probability measure preserving action of a countable group such that $\mathcal{F}(\mathcal{R}(G \curvearrowright X))=\{1\}$ and $\mathcal{F}(\mathrm{L}^\infty(X)\rtimes G)= \mathbb{
R}^*_+$, so the biggest possible difference can be realized. Earlier, Popa found free ergodic probability measure preserving actions for which $\mathcal{F}(\cR(G \actson X))$ is countable, while the II$_1$ factor $\rL^\infty(X) \rtimes G$ has fundamental group $\mathbb{R}^*_+$ (\cite[Corollary of Theorem 3]{Po90} and \cite[Corollary 4.7.2]{Po86}).

The question we address is in some sense ``opposite'' to these results of Popa. We construct a group measure space II$_1$ factor $M = \rL^\infty(X) \rtimes G$ with trivial fundamental group, but admitting a Cartan subalgebra $A$ that is non-conjugate to $\rL^\infty(X)$ and for which the associated equivalence relation has non-trivial fundamental group. By the previous paragraphs, the equivalence relation associated to the second Cartan inclusion $A \subset M$ must come with a non-trivial 2-cocycle. Indeed, otherwise its fundamental group would be a subset of $\mathcal{F}(M)$. Therefore our results are closely related to earlier constructions of II$_1$ factors with several Cartan subalgebras.

The first example of a II$_1$ factor with two Cartan subalgebras that are not conjugate by an automorphism was given by Connes and Jones in \cite{CJ82}. In \cite[\S 7]{OP08}, Ozawa and Popa gave more explicit examples of this phenomenon. In fact, a II$_1$ factor can have uncountably many non-conjugate Cartan subalgebras (see \cite{Po86}, \cite{Po90}, \cite{Po06} and \cite{SV11}).

The concrete II$_1$ factor $M$ that we construct in this paper is an amalgamated free product (AFP). Several rigidity theorems, including computations of fundamental groups, for such AFP II$_1$ factors were obtained in \cite{IPP05}. Our proof that $\mathcal{F}(M)=\{1\}$ uses the techniques of \cite{IPP05} and the recent work of \cite{Io12} on Cartan subalgebras in AFP II$_1$ factors.

\subsection*{Acknowledgement}

The authors are grateful to Stefaan Vaes for his suggestions and many fruitful discussions.

\subsection*{Statements of the main results}

We first fix the setting of the example.

\begin{notation}\label{notation}
Let $n \geq 6$ even and define $\Sigma$ as $$\Sigma:=\left(\begin{matrix}
\SLtwo & 0 &  \dots & 0\\
0 & \SLtwo & \dots & 0\\
\vdots & \vdots & \ddots & \vdots\\
0 & 0 & \dots & \SLtwo
\end{matrix}\right) < \SLn \;.
$$
Set $$\Gamma = \SLn \ast_{\Sigma} (\Sigma \times \Lambda) \; ,$$ where $\Lambda$ is a countable infinite group. Let $q: \Gamma \to \SLn$ be the quotient map, i.e. $q(g)=g$ for all $g \in \SLn$ and $q(\lambda)=e$ for all $\lambda \in \Lambda$. Fix a prime $p$ and consider the action $\Gamma \overset{\alpha}{\curvearrowright} \mathbb{Z}_p^n$ through $q$. Remark that this action preserves $\mathbb{Z}^n$. Write $G := \mathbb{Z}^n \rtimes \Gamma$.

Define $X := (\mathbb{Z}_p^n)^\Gamma$. Embed $\mathbb{Z}_p^n$ into $X$ by $i: \mathbb{Z}_p^n \to X: z \mapsto (\alpha_g(z))_{g \in \Gamma}$. Now consider the action $G \curvearrowright X$, where $\mathbb{Z}^n$ acts by translation after embedding by $i$, and $\Gamma$ acts by Bernoulli shift. One can see that this is a free ergodic probability measure preserving action. Throughout this paper we denote by $M$
$$M := \mathrm{L}^\infty(X) \rtimes G = \rL^\infty((\mathbb{Z}_p^n)^\Gamma)\rtimes (\mathbb{Z}^n \rtimes (\SLn \ast_{\Sigma} (\Sigma \times \Lambda)))$$ the associated group measure space II$_1$ factor.
\end{notation}
We will prove that $M$ has a Cartan subalgebra that is non-conjugate to $\rL^\infty(X)$. We describe the associated equivalence relation, that comes with a non-trivial 2-cocycle. We compute its fundamental group as well as $\mathcal{F}(\mathcal{R}(G \curvearrowright X))$.

\begin{theorem}
With $G \curvearrowright X$ as in Notation \ref{notation}, $M:=\rL^\infty(X)\rtimes G$ has at least two Cartan subalgebras $A_1=\mathrm{L}^\infty(X)$ and $A_2=\mathrm{L}^\infty\left(\frac{(\mathbb{Z}_p^n)^\Gamma}{i(\mathbb{Z}_p^n)}\right) \ovt \mathcal{L}(\mathbb{Z}^n)$ that are not conjugate by an automorphism of $M$. Denote the associated equivalence relations by $\mathcal{R}_1$ and $\mathcal{R}_2$. Then the following holds:\begin{enumerate}
\item $\mathcal{F}(M)=\mathcal{F}(\mathcal{R}_1)=\{1\}  ,$
\item $\mathcal{F}(\mathcal{R}_2)=\{p^{kn}\mid k \in \mathbb{Z}\} \; .$
\end{enumerate}
\end{theorem}

This theorem follows from Theorem \ref{triv}, Theorem \ref{twocart} and Theorem \ref{nontriv}.

\section{Preliminaries}

\subsection{Group actions on measure spaces}

\begin{definition}{\cite[\S 1.5]{Po03}}
Let $G$ be a locally compact second countable group and $G \curvearrowright (X,\mu)$ a Borel
action preserving the finite or infinite measure $\mu$. The action is called \textit{s-malleable} if there exists
\begin{itemize}
\item a one-parameter group $(\alpha_t)_{t\in \mathbb{R}}$ of measure preserving transformations of $X \times X$,
\item an involutive measure preserving transformation $\beta$ of $X \times X$,
\end{itemize}
such that
\begin{itemize}
\item $\alpha_t$ and $\beta$ commute with the diagonal action $G \curvearrowright X \times X$,
\item $\alpha_1(x,y) \in \{y\} \times X$ for almost all $(x,y) \in X \times X$,
\item $\beta(x,y) \in \{x\} \times X$ for almost all $(x,y) \in X \times X$,
\item $\alpha_t \circ \beta = \beta \circ \alpha_{-t}$ for all $t \in \mathbb{R}$.
\end{itemize}
\end{definition}

\begin{definition}
A non-singular action $G \curvearrowright (X,\mu)$ of a locally compact second countable group $G$ on a standard measure space $(X,\mu)$ is called \textit{essentially free and proper} if there exists a measurable map $\pi: X \to G$ such that $\pi(g\cdot x)= g\pi(x)$ for almost all $(g,x) \in G \times X$.
\end{definition}

\begin{definition}\label{SOE}
Let $\Gamma \stackrel{\alpha}{\curvearrowright} (X,\mu)$ and $\Lambda \stackrel{\beta}{\curvearrowright} (Y,\nu)$ be essentially free, ergodic, non-singular actions of countable groups on standard measure spaces.
\begin{itemize}
\item Let $\delta: \Gamma \rightarrow \Lambda$ be a group isomorphism. A $\delta$-conjugacy of the actions $\alpha$ and $\beta$ is a non-singular isomorphism $\Delta: X \rightarrow Y$ such that $\Delta(g \cdot x) = \delta(g) \cdot \Delta(x)$ for all $g \in \Gamma$ and a.e. $x \in X$.
\item A \textit{stable orbit equivalence} between the actions $\alpha$ and $\beta$ is a non-singular isomorphism $\Delta : X_0 \to Y_0$ between non-negligible subsets $X_0\subset X, Y_0\subset Y$, such that $\Delta$ is an isomorphism between the restricted orbit equivalence relations $\mathcal{R}(\Gamma \curvearrowright X)_{\mid X_0}$ and $\mathcal{R}(\Lambda \curvearrowright Y)_{\mid Y_0}$. The compression constant of $\Delta$ is defined as $c(\Delta) := \frac{\nu(Y_0)}{\mu(X_0)}$.
\end{itemize}
\end{definition}

\begin{remark}\label{remark-extend-SOE}
Let $\Gamma \stackrel{\alpha}{\curvearrowright} (X,\mu)$ and $\Lambda \stackrel{\beta}{\curvearrowright} (Y,\nu)$ be essentially free, ergodic, probability measure preserving actions of countable groups on standard measure spaces.
By definition, a stable orbit equivalence between $\alpha$ and $\beta$ is a measure space isomorphism $\Delta : X_0 \recht Y_0$ between non-negligible subsets $X_0 \subset X, Y_0 \subset Y$, satisfying
$$\Delta ( \Gamma \cdot x \cap X_0) = \Lambda \cdot \Delta(x) \cap Y_0$$
for a.e.\ $x \in X_0$. By ergodicity of $\Gamma \actson X$, we can choose a measurable map $\Theta : X \recht X_0$ satisfying $\Theta(x) \in \Gamma \cdot x$ for a.e.\ $x \in X$. Denote $\Delta_0 := \Delta \circ \Theta$. By construction $\Delta_0$ is a local isomorphism from $X$ to $Y$. This means that $\Delta_0 : X \recht Y$ is a Borel map and that $X$ can be partitioned into a sequence of non-negligible subsets $\cW \subset X$, such that the restriction of $\Delta_0$ to any of these subsets $\cW$ is a measure space isomorphism of $\cW$ onto some non-negligible subset of $Y$. Also by construction $\Delta_0$ is orbit preserving, meaning that for a.e.\ $x,y \in X$ we have that $x \in \Gamma \cdot y$ iff $\Delta_0(x) \in \Lambda \cdot \Delta_0(y)$.
\end{remark}

\begin{definition}
Let $\Gamma \stackrel{\alpha}{\curvearrowright} (X,\mu)$ be an essentially free, ergodic, non-singular action of a countable group on a standard measure space. We say that $\alpha$ is \textit{induced} from $\Gamma_1 \curvearrowright X_1$, if $\Gamma_1$ is a subgroup of $\Gamma$, $X_1$ is a non-negligible subset of $X$ and $g \cdot X_1 \cap X_1$ is negligible for all $g \in \Gamma - \Gamma_1$.
\end{definition}

\subsection{Fundamental groups}

The \textit{fundamental group} $\mathcal{F}(M)$ of a II$_1$ factor $(M,\tau)$, introduced in \cite{MvN43}, is defined as the following subgroup of $\mathbb{R}^*_+$.
$$\mathcal{F}(M)=\left\{\frac{\tau(p)}{\tau(q)} \;\middle\vert\; p,q \mbox{ are non-zero projections in }M\mbox{ such that }pMp\cong qMq\right\}.$$
An ergodic probability measure preserving measurable equivalence relation with countable equivalence classes on a standard probability space $(X,\mu)$ is called a II$_1$ equivalence relation. The fundamental group $\mathcal{F}(\mathcal{R})$ of a II$_1$ equivalence relation $\mathcal{R}$ is defined as
$$\mathcal{F}(\mathcal{R})=\left\{\frac{\mu(Y)}{\mu(Z)}\;\middle\vert\; \mathcal{R}_{\mid Y} \cong \mathcal{R}_{\mid Z}\right\}.$$

\subsection{Different types of mixing}

\begin{definition}{\cite[\S 2]{Sch84}} \label{def-sch}
Let $G$ be a countable group, $(X,\mu)$ a standard probability space and $G \curvearrowright X$ a probability measure preserving, ergodic action. We say that $G \curvearrowright X$ is \begin{enumerate}
\item \textit{weakly mixing} if for every $\epsilon > 0$, every $n \in \mathbb{N}$ and all measurable $A_1,\ldots,A_n \subset X$ there is a $g \in G$ such that for all $i , j \in \{1,\ldots,n\}$ we have
$$|\mu(A_i \cap gA_j)-\mu(A_i)\mu(A_j)|<\epsilon \; .$$
\item \textit{mildly mixing} if for every measurable set $B \subset X$ with $0 < \mu(B) < 1$,
    $$\liminf_{g \to \infty} \mu(B \Delta gB)>0 \;, $$
\item \textit{strongly mixing} if for all measurable sets $A,B \subset X$ we have $$\lim_{g \to \infty} \mu(A \cap gB)= \mu(A)\mu(B) \;.$$
\end{enumerate}
\end{definition}
Examples of actions satisfying the strongest of these conditions include all Bernoulli actions of countable groups.
\begin{proposition}{\cite[Proposition 2.3]{Sch84}}\label{mildly}
Let $G \curvearrowright (X,\mu)$ be a measure preserving, ergodic action of a countable group on a standard probability space. Then $G \curvearrowright (X,\mu)$ is mildly mixing if and only if for every non-singular, properly ergodic (i.e. every $G$-orbit has measure zero) action $G \curvearrowright (Y,\nu)$ on a $\sigma$-finite standard measure space, the product action $G \curvearrowright (X \times Y, \mu \times \nu)$ is ergodic.
\end{proposition}

If we restrict $G \curvearrowright (Y,\nu)$ to probability measure preserving actions on standard probability spaces, Proposition \ref{mildly} gives a characterization of weakly mixing actions. It is now clear that $3 \Rightarrow 2 \Rightarrow 1$ in Definition \ref{def-sch}.

\subsection{Cocycle superrigidity}

\begin{definition}{\cite[Definition 2.5]{Po05}}
 A Polish group is \textit{of finite type} if it can be realized as a closed subgroup of the unitary
group of some II$_1$ factor with separable predual.
\end{definition}
All countable and all second countable compact
groups are Polish groups of finite type.

\begin{definition}{\cite[Definition 2.5]{Po05}}
A non-singular action $G \curvearrowright (X,\mu)$ of a locally compact second countable
group $G$ on a standard measure space $(X,\mu)$ is called \textit{$\mathcal{U}_{fin}$-cocycle superrigid} if every 1-cocycle
for the action $G \curvearrowright (X,\mu)$ with values in a Polish group of finite type $\mathcal{G}$ is cohomologous to a continuous group morphism $G \to \mathcal{G}.$
\end{definition}

The following is a slightly different version of \cite[Proposition 3.6 (2)]{Po05} (see also \cite[Lemma 3.5]{Fu06}).

\begin{lemma}\label{cocycle}
Let $G \curvearrowright (X,\mu)$ be a non-singular action of a countable group $G$ on a standard measure space $(X,\mu)$. Let $ \omega: G \times X \to \mathcal{G}$ be a 1-cocycle with values in the Polish group $\mathcal{G}$ with a bi-invariant metric. Let $H < G$ be a subgroup and assume that $\omega(h,x)=\delta(h)$ for all $h \in H$ and a.e. $x \in X$, where $\delta: H \to \mathcal{G}$ is a group morphism. For any $g_0 \in G$ such that the diagonal action $H_0 = H \cap g_0^{-1}Hg_0 \curvearrowright X \times X$ is ergodic, $x \mapsto \omega(g_0,x)$ is essentially constant.
\end{lemma}
\begin{proof}
For any $h \in H_0$ denote $\alpha(h) = g_0 h g_0^{-1}$ and remark that $h, \alpha(h) \in H$. For all $h \in H_0$ and a.e. $x \in X$ we have
\begin{eqnarray*}
\delta(\alpha(h))\omega(g_0,x) & = & \omega(\alpha(h),g_0\cdot x)\omega(g_0,x)\\
& = & \omega(\alpha(h)g_0,x)\\
& = & \omega(g_0h,x)\\
& = & \omega(g_0,h\cdot x)\omega(h,x)\\
& = & \omega(g_0,h\cdot x)\delta(h) \;,
\end{eqnarray*}
so $\omega(g_0,h\cdot x) = \delta(\alpha(h))\omega(g_0,x)\delta(h^{-1})$. Now consider the map $\varphi:X \times X \to \mathbb{R}: (x,y) \mapsto d(\omega(g_0,x), \omega(g_0,y))$. Then $\varphi$ is essentially invariant under the diagonal action $H_0 \curvearrowright X \times X$, as the metric is bi-invariant. But as $H_0 \curvearrowright X \times X$ is ergodic, $\varphi$ has to be essentially constant, hence 0. This proves that $x \mapsto \omega(g_0,x)$ is essentially constant.
\end{proof}

\section{The fundamental group of $M$ is trivial}

\begin{theorem}\label{triv}
With $G \curvearrowright X$ and $M = \rL^\infty(X) \rtimes G$ as in Notation \ref{notation}, we have $\mathcal{F}(M)=\{1\}$. Hence also $\mathcal{F}(\mathcal{R}(G \curvearrowright X))=\{1\}.$
\end{theorem}

We first prove two general lemmas.

\begin{definition}
Let $p$ be a projection in a von Neumann algebra $M$. Denote by $\mathcal{Z}(M)$ the center of $M$ and by $z_M(p)$ the central support of $p$ in $M$. We define the $*$-isomorphism
$$\vphi_p: \mathcal{Z}(M) z_M(p) \rightarrow \mathcal{Z}(pMp) \quad\text{by}\;\; \vphi_p(m) = mp \; .$$
\end{definition}

\begin{lemma}\label{lemma.cornerB}
Let $K$ be a compact abelian group with countable dense subgroup $Z$. Let $\Gamma_1$ and $\Gamma_2$ be two countable groups with a common subgroup $\Sigma$ such that $[\Gamma_1 : \Sigma] \geq 2$ and $[\Gamma_2 : \Sigma] \geq 3$. Denote $\Gamma = \Gamma_1 \underset{\Sigma}{*} \Gamma_2$ and suppose that there exist $g_1, g_2, \ldots, g_n \in \Gamma$ such that $\cap_{i=1}^n g_i \Sigma g_i^{-1}$ is finite. Assume that $\Gamma$ acts on $K$ by continuous group automorphisms $(\alpha_g)_{g \in \Gamma}$ preserving $Z$. Embed $K$ in $K^\Gamma$ by $i: K \to K^\Gamma: k \mapsto (\alpha_g(k))_{g \in \Gamma}$. Write $M := \rL^\infty(K^\Gamma) \rtimes (Z \rtimes \Gamma)$ where $Z$ acts on $K^\Gamma$ by translation after embedding by $i$ and $\Gamma$ acts by Bernoulli shift.

Let $p$ be a projection in $M$ and let $\alpha: M \rightarrow pMp$ be a stable automorphism of $M$. Denote $B := \rL^\infty(K^\Gamma) \rtimes Z$. Remark that $\mathcal{Z}(B) = \rL^\infty(K^\Gamma/i(K))$.
\begin{enumerate}
\item There exist projections $q,r \in B$ such that after composition of $\alpha$ with an inner automorphism of $M$, we have $\alpha(q) = r$ and $\alpha(qBq)=rBr$.
\item Furthermore, the $*$-isomorphism $\Psi: \mathcal{Z}(B)z_B(q) \rightarrow \mathcal{Z}(B)z_B(r)$ given by the composition of
\begin{equation*}
\begin{tabular}{ccccccc}
$\mathcal{Z}(B)z_B(q)$ & $\overset{\vphi_q}{\rightarrow}$ & $\mathcal{Z}(q B q)$ & $\overset{\alpha}{\rightarrow}$ & $\mathcal{Z}(r B r)$ & $\overset{\vphi_r^{-1}}{\rightarrow}$ & $\mathcal{Z}(B) z_B(r)$
\end{tabular}
\end{equation*}
is a stable orbit equivalence between the action $\Gamma \actson K^\Gamma/i(K)$ and itself.
\end{enumerate}
\end{lemma}
\begin{proof}
Denote $X = K^\Gamma$. Note that $B$ is amenable and that $M=B\rtimes \Gamma$. By \cite[Theorem 7.1]{Io12} we get that $\alpha(B) \embed{M} B$ and that $B \embed{M} \alpha(B)$. Remark that $B^\prime \cap M = \rL^\infty(X/i(K)) =  \mathcal{Z}(B)$. \cite[Lemma 3.5]{Va07} then gives
\begin{equation}\label{eq.both-embed}
\mathcal{Z}(B) \embed{M} \alpha(\mathcal{Z}(B)) \quad\text{and}\quad \alpha(\mathcal{Z}(B)) \embed{M} \mathcal{Z}(B) \; .
\end{equation}
Since $\mathcal{Z}(B)$ is regular in $M$ and $\alpha(\mathcal{Z}(B))$ is regular in $pMp$, (\ref{eq.both-embed}) implies that $\rL^2(M) p$ can be written as a direct sum of $\mathcal{Z}(B)$-$\alpha(\mathcal{Z}(B))$-bimodules with dim$(-_{\alpha(\mathcal{Z}(B))})$ finite and that $p \rL^2(M)$ can be written as a direct sum of $\alpha(\mathcal{Z}(B))$-$\mathcal{Z}(B)$-bimodules with dim$(-_{\mathcal{Z}(B)})$ finite.

So $\rL^2(M) p$ admits a non-zero $\mathcal{Z}(B)$-$\alpha(\mathcal{Z}(B))$-subbimodule with dim$(_{\mathcal{Z}(B)}-)$ finite and dim$(-_{\alpha(\mathcal{Z}(B))})$ finite. This means that there exists
\begin{itemize}
\item a projection $s \in M_n(\C) \otimes \alpha(\mathcal{Z}(B))$
\item a non-zero partial isometry $v \in (M_{1,n}(\C) \otimes M) s$
\item a unital $*$-homomorphism $\theta: \mathcal{Z}(B) \rightarrow s(M_n(\C) \otimes \alpha(\mathcal{Z}(B)))s$
\end{itemize}
such that $\theta(\mathcal{Z}(B)) \subset s(M_n(\C) \otimes \alpha(\mathcal{Z}(B)))s$ has finite index and $b v = v \theta(b)$ for all $b \in \mathcal{Z}(B)$. Manipulating $s,v$ and $\theta$, we obtain projections $z \in \mathcal{Z}(B), t \in \alpha(\mathcal{Z}(B))$, a non-zero partial isometry $v \in M t$ and a unital $*$-isomorphism $\theta: \mathcal{Z}(B)z \rightarrow \alpha(\mathcal{Z}(B))t$ such that $b v = v \theta(b)$ for all $b \in \mathcal{Z}(B)z$.

Remark that $\mathcal{Z}(B)^\prime \cap M = B$, $vv^* \in Bz$ and $v^*v \in \alpha(B)t$. One verifies that $v \alpha(\mathcal{Z}(B)) v^* = vv^* \mathcal{Z}(B) vv^*$. By taking relative commutants in $vv^* M vv^*$, it follows that $v \alpha(B) v^* = vv^* B vv^*$. Extend $v$ to a unitary $u \in M$. Define $q := \alpha^{-1}(v^*v) \in B$ and $r := vv^* \in B$. Then $(\Ad u \circ \alpha)(q) = r$ and $(\Ad u \circ \alpha)(qBq) = rBr$.

We now prove the second part of the lemma. For any projection $q \in B$ define the set $\mathcal{I}_q$ as follows:
\begin{equation*}
\mathcal{I}_q := \{v \in qMq \mid v \;\text{is a partial isometry with} \; v^* v, v v^* \in qBq \;\text{and} \; vBv^* = v v^* B v v^*\} \; .
\end{equation*}
For every $v \in \mathcal{I}_q$ there exists a unique $*$-isomorphism $\theta_v: \mathcal{Z}(B)z_B(v^* v) \rightarrow \mathcal{Z}(B)z_B(v v^*)$ satisfying
\begin{equation*}
\theta_v(b) v = v b \quad\text{for all}\quad b \in \mathcal{Z}(B)z_B(v^* v) \; .
\end{equation*}
Denote by $\mathcal{Q} \subset X/i(K)$ the support of the projection $z_B(q)$. By construction the restricted orbit equivalence relation $\cR(\Gamma \actson X/i(K))_{|\mathcal{Q}}$ is generated by the graphs of $\theta_v$ with $v \in \mathcal{I}_q$. To conclude the proof it suffices to remark that $\alpha(\mathcal{I}_q) = \mathcal{I}_r$.
\end{proof}

Recall that an infinite subgroup $H$ of a group $\Gamma$ is wq-normal in $\Gamma$ if there exists an increasing sequence $(H_n)_n$ of subgroups of $\Gamma$ with $H_0 = H, \cup_n H_n = \Gamma$ and such that for all $n$ the group $H_n$ is generated by the elements $g \in \Gamma$ with $|g H_{n-1} g^{-1} \cap H_{n-1}| = \infty$.

An inclusion of groups $H \subset \Gamma$ has the relative property (T) of Kazhdan-Margulis if any unitary representation of $\Gamma$ that almost contains the trivial representation of $\Gamma$ must contain the trivial representation
of $H$. We call such $H$ a rigid subgroup of $\Gamma$.

\begin{lemma}\label{cocyclesuperrigid}
Let $K$ be a compact group and let $\Gamma$ be a countable group. Assume that $\Gamma$ admits an infinite rigid subgroup that is wq-normal in $\Gamma$. Assume that $\Gamma$ acts on $K$ by continuous group automorphisms $(\alpha_g)_{g \in \Gamma}$. Embed $K$ in $K^\Gamma$ by $i: K \to K^\Gamma: k \mapsto (\alpha_g(k))_{g \in \Gamma}$. Then the action $K \rtimes \Gamma \actson K^\Gamma$ where $K$ acts by translation after embedding by $i$ and $\Gamma$ acts by Bernoulli shift, is $\mathcal{U}_{fin}$-cocycle superrigid.
\end{lemma}
\begin{proof}
Let $\omega: (K \rtimes \Gamma) \times K^\Gamma \rightarrow \mathcal{G}$ be a 1-cocycle for the action $K \rtimes \Gamma \actson K^\Gamma$ with values in a Polish group of finite type. By Popa's cocycle superrigidity theorem \cite[Theorem 0.1]{Po05}
we may assume that the restriction $\omega: \Gamma \times K^\Gamma \rightarrow \mathcal{G}$ is a group morphism $\delta$, i.e. $\omega(g,x) = \delta(g)$ for all $g \in \Gamma$ and a.e. $x \in K^\Gamma$. It remains to prove that $\omega_{| K \times K^\Gamma}$ is a group morphism.

For all $g \in \Gamma, k \in K$ and a.e. $x \in K^\Gamma$ we have
\begin{eqnarray*}
\omega(\alpha_g(k), g \cdot x) \delta(g) & = & \omega(\alpha_g(k), g \cdot x) \omega(g,x) \\
& = & \omega(\alpha_g(k) g, x) \\
& = & \omega(g k,x) \\
& = & \omega(g, k \cdot x) \omega(k,x) \\
& = & \delta(g) \omega(k,x)
\end{eqnarray*}
so that $\omega(\alpha_g(k), g \cdot x) = \delta(g) \omega(k,x) \delta(g)^{-1}$. Applying \cite[Lemma 5.4]{PV08b} to the restriction $\omega: K \times K^\Gamma \rightarrow \mathcal{G}$ then gives that $\omega_{| K \times K^\Gamma}$ is essentially independent of $K^\Gamma$. It follows that $\omega$ is a group morphism.
\end{proof}

Remark that for $\Gamma = \SLn \ast_{\Sigma} (\Sigma \times \Lambda)$ as in Notation \ref{notation}, $\SLn$ is an infinite rigid subgroup that is wq-normal in $\Gamma$. We now prove Theorem \ref{triv}.

\begin{proof}[Proof of Theorem \ref{triv}.]
Let $p$ be a projection in $M$ and let $\alpha: M \rightarrow pMp$ be a stable automorphism of $M$. We will prove that $p = 1$, i.e. $\alpha$ is an automorphism of $M$. This means that $\mathcal{F}(M) =\{1\}$.

Denote $K := \mathbb{Z}_p^n$ and $Z := \mathbb{Z}^n$. Recall that $X = K^\Gamma$ with $\Gamma$ as in Notation \ref{notation} and that $M = \rL^\infty(X) \rtimes (Z \rtimes \Gamma)$ where $Z$ acts by translation after embedding by $i$ and $\Gamma$ acts by Bernoulli shift. Denote $B := \rL^\infty(X) \rtimes Z$. Remark that $\mathcal{Z}(B) = \rL^\infty(X/i(K))$. By Lemma \ref{lemma.cornerB} there exist projections $q,r \in B$ such that after composition of $\alpha$ with an inner automorphism of $M$, we have $\alpha(q) = r$ and $\alpha(qBq)=rBr$. Furthermore $\Psi_0: \mathcal{Z}(B)z_B(q) \rightarrow \mathcal{Z}(B) z_B(r)$ given by the composition of
\begin{equation*}
\begin{tabular}{ccccccc}
$\mathcal{Z}(B)z_B(q)$ & $\overset{\vphi_q}{\rightarrow}$ & $\mathcal{Z}(q B q)$ & $\overset{\alpha}{\rightarrow}$ & $\mathcal{Z}(r B r)$ & $\overset{\vphi_r^{-1}}{\rightarrow}$ & $\mathcal{Z}(B) z_B(r)$
\end{tabular}
\end{equation*}
is a stable orbit equivalence of the action $\Gamma \actson X/i(K)$. To simplify notation in the rest of the proof, we will use $z(q), z(r)$ for $z_B(q), z_B(r)$ respectively.

Denote by $\mathcal{Q}, \mathcal{R} \subset X/i(K)$ the support of $z(q), z(r)$ respectively. Let $\Delta_0: \mathcal{R} \rightarrow \mathcal{Q}$ be the measure space isomorphism such that $\Psi_0(b) = b \circ \Delta_0$ for all $b \in Z(B)z(q) = \rL^\infty(\mathcal{Q})$. By ergodicity of the action $\Gamma \actson X/i(K)$, we can extend $\Delta_0$ to a local isomorphism from $X/i(K)$ to $X/i(K)$ that is orbit preserving, as explained in Remark \ref{remark-extend-SOE}.

Let $G_1$ be the locally compact second countable group $K \rtimes \Gamma$, having $K$ as a compact open normal subgroup. By Lemma \ref{cocyclesuperrigid} the action $G_1 \overset{\si}{\actson} X$ where $K$ acts by translation after embedding by $i$ and $\Gamma$ acts by Bernoulli shift, is $\mathcal{U}_{fin}$-cocycle superrigid. Remark that the restricted action $\si_{|K}$ is proper by compactness of $K$ and that $G_1/K \actson X/i(K)$ is the action $\Gamma \actson X/i(K)$. 

Since the action $\Gamma \actson X/i(K)$ is mixing, it is not induced from a proper subgroup. Applying \cite[Lemma 5.10]{PV08b} to the stable orbit equivalence $\Delta_0$ between the action $\Gamma \actson X/i(K)$ and itself, we find an open normal subgroup $K_1 \triangleleft G_1$ such that the following holds.
\begin{enumerate}
\item[(i)] The restricted action $\si_{|K_1}$ is proper. 
\item[(ii)] The actions $G_1/K_1 \actson X/K_1$ and $\Gamma \actson X/i(K)$ are conjugate through a non-singular isomorphism $\Delta: X/K_1 \rightarrow X/i(K)$ and a group isomorphism $\delta: G_1/K_1 \rightarrow \Gamma$.
\item[(iii)] $\Delta_0(i(K) \cdot x) \in \Gamma \cdot \Delta(K_1 \cdot x)$ for almost all $x \in X$.
\end{enumerate}
Since the restricted action $K_1 \actson X$ is essentially free and proper, there exists a measurable map $\pi: X \rightarrow K_1$ such that $\pi(k \cdot x) = k \pi(x)$ for almost all $(k,x) \in K_1 \times X$. Then the pushforward of the invariant probability measure on $X$ is an invariant probability measure on $K_1$. So $K_1$ is compact. Since $K_1$ is a compact normal subgroup of $G_1 = K \rtimes \Gamma$, the image of $K_1$ in $\Gamma$ is a finite normal subgroup of $\Gamma$. Being an icc  group (i.e. a group with infinite conjugacy classes), $\Gamma$ has no non-trivial finite normal subgroups. It follows that the image of $K_1$ in $\Gamma$ is trivial, so that $K_1 \subset K$. 

As $K_1$ is open and $K$ is compact, $K_1$ is a finite index subgroup of $K$. So $K/K_1$ is a finite normal subgroup of $G_1/K_1$. By (ii) the group $G_1/K_1$ is isomorphic to $\Gamma$, hence $G_1/K_1$ has no non-trivial finite normal subgroups. But then $K/K_1$ is trivial, so $K_1 = K$. It follows that $\delta$ is a group automorphism of $\Gamma$ and that $\Delta: X/i(K) \rightarrow X/i(K)$ is a $\delta$-conjugacy between the action $\Gamma \actson X/i(K)$ and itself, satisfying
\begin{equation}
\Delta_0(i(K) \cdot x) \in \Gamma \cdot \Delta(i(K) \cdot x) \label{similar-SOE}
\end{equation}
for almost all $x \in X$. Remark that $\Delta$ is a measure space isomorphism, as any conjugacy between ergodic probability measure preserving actions is measure preserving. In particular, the compression constant of $\Delta_0$ is 1.

Then, by ergodicity of the action $\Gamma \actson X/i(K)$, one can build a measure space isomorphism $\tilde{\Delta}_0: X/i(K) \rightarrow X/i(K)$ such that $\tilde{\Delta}_{0|\mathcal{R}} = {\Delta_0}_{|\mathcal{R}}$ and $\tilde{\Delta}_0(i(K) \cdot x) \in \Gamma \cdot \Delta_0(i(K) \cdot x)$ for almost all $x \in X$. In particular, $\tilde{\Delta}_0$ still satisfies (\ref{similar-SOE}) for almost all $x \in X$. 

It follows that there exists a unitary $u \in \rL^\infty(X/i(K)) \rtimes \Gamma$ such that $u (b \circ \tilde{\Delta}_0) u^* = b \circ \Delta$ for all $b \in \rL^\infty(X/i(K))$. Recall that the stable orbit equivalence $\vphi_r^{-1} \circ \alpha \circ \vphi_q$ of the action $\Gamma \actson X/i(K)$ is given by $(\vphi_r^{-1} \circ \alpha \circ \vphi_q)(b) = b \circ \Delta_0 = b \circ \tilde{\Delta}_0$ for all $b \in Z(B)z(q)$. Replacing $r$ by $uru^*$ and $\alpha$ by $\Ad u \circ \alpha$, we then find that $(\vphi_r^{-1} \circ \alpha \circ \vphi_q)(b) = b \circ \Delta$ for all $b \in Z(B)z(q)$. 

To summarize, we found a group automorphism $\delta$ of $\Gamma$ and a $\delta$-conjugacy $\Delta$ of the action $\Gamma \actson X/i(K)$ such that $(\vphi_r^{-1} \circ \alpha \circ \vphi_q)(b) = b \circ \Delta$ for all $b \in Z(B)z(q)$. For convenience, we denote $\Psi(b) = b \circ \Delta$ for all $b \in \rL^\infty(X/i(K))$.

Let $\tau$ be the unique tracial state on $M$. We show that after composition with an inner automorphism of $M$, $\alpha$ satisfies
\begin{equation}\label{eq.alpha-almostconj}
\alpha(b) = \Psi(b)p_0  \quad\text{for all}\quad b \in \mathcal{Z}(B)z(q) \;,
\end{equation}
where $p_0$ is a projection in $B$ of trace $\tau(p) \tau(z(q))$. The proof makes use of the following equality. For every projection $q_0 \leq q$ we have
\begin{equation}\label{eq.condexp-alpha}
E_{\mathcal{Z}(B)}(\alpha(q_0)) = \tau(p) \Psi(E_{\mathcal{Z}(B)}(q_0)) \; .
\end{equation}
Here $E_{\mathcal{Z}(B)}$ denotes the unique trace preserving conditional expectation of $M$ onto $\mathcal{Z}(B)$.
To see that (\ref{eq.condexp-alpha}) holds, first remark that $\alpha(xq) = \Psi(x)r$ for all $x \in \mathcal{Z}(B)$. Then use the fact that $\alpha$ is $\tau$-scaling ($\tau \circ \alpha = \tau(p) \tau$) and that $E_{\mathcal{Z}(B)}$ and $\Psi$ are $\tau$-preserving to show that
\begin{eqnarray*}
\tau( \Psi(x)  E_{\mathcal{Z}(B)}(\alpha(q_0)) ) & = & \tau( \Psi(x)  \alpha(q_0) ) \\
& = & \tau(\alpha(x q_0)) \\
& = & \tau(p) \tau(x q_0) \\
& = & \tau(p) \tau(E_{\mathcal{Z}(B)}(x q_0)) \\
& = & \tau(p) \, \tau( \Psi(x)  \Psi(E_{\mathcal{Z}(B)}(q_0)) )
\end{eqnarray*}
for all $x \in \mathcal{Z}(B)$. Formula (\ref{eq.condexp-alpha}) follows.

We will show that there exist partial isometries $v_n, w_n \in B$ satisfying the following properties:
\begin{equation}\label{eq.partisometries}
v_n^* v_n \leq q \;\; , \;\; \alpha(v_n^* v_n) = w_n^* w_n \;\; , \;\; \sum_n v_n v_n^* = z(q) \;\; , \;\; w_n w_n^* \;\; \text{mutually orthogonal} \; .
\end{equation}
Then define $p_0 := \sum w_n w_n^*$. Note that $p_0$ is a projection in $B$ of trace $\tau(p) \tau(z(q))$. Define $v := \sum_n w_n \alpha(v_n^*) \in M$. Remark that $vv^* = p_0$ and $v^*v = \alpha(z(q))$. Extend $v$ to a unitary $u \in M$. One verifies that $(\Ad u \circ \alpha)(b) = \Psi(b) p_0$ for all $b \in \mathcal{Z}(B) z(q)$ so that (\ref{eq.alpha-almostconj}) is shown.

It remains to prove the existence of the partial isometries in (\ref{eq.partisometries}). Consider the set
\begin{eqnarray*}
\mathcal{J} = \bigl\{ \; \bigl\{ \; (v_n, w_n) \in B \times B  \;\big|\; v_n^* v_n \leq q, \; \alpha(v_n^* v_n) = w_n^* w_n, \;\\ v_n v_n^* \; \text{mutually} \perp, \; w_n w_n^* \; \text{mutually} \perp \; \bigr\} \; \bigr\} \; .
\end{eqnarray*}
Remark that $\{(q,r)\} \in \mathcal{J}$. The set $\mathcal{J}$ is partially ordered by inclusion. By Zorn's lemma, take a maximal element $\{(v_n,w_n) \mid n \}$ of $\mathcal{J}$. It suffices to show that $\sum v_n v_n^* = z(q)$. Assume that $\sum v_n v_n^* < z(q)$. Then there exists a partial isometry $v \in B$ such that $v^*v \leq q$ and $vv^* \leq z(q) - \sum v_n v_n^*$. We claim that 
\begin{equation}\label{eq.subeq}
\alpha(v^* v) \prec z(r) - \sum w_n w_n^* \; ,
\end{equation}
where $\prec$ refers to the comparison of projections in $B$.
So there exists $w \in B$ such that $w^* w = \alpha(v^* v)$ and $w w^* \leq z(r) - \sum w_n w_n^*$. This means that we can add $(v,w)$ to the family $\{(v_n,w_n) \mid n \}$, contradicting its maximality. It remains to prove (\ref{eq.subeq}). Using (\ref{eq.condexp-alpha}) we find that
$$ E_{\mathcal{Z}(B)}(\alpha(v^* v)) = \tau(p) \Psi(E_{\mathcal{Z}(B)}(v^* v)) = \Psi(E_{\mathcal{Z}(B)}(\tau(p)  v v^*)) \; . $$
On the other hand
\begin{eqnarray*}
E_{\mathcal{Z}(B)} \big(z(r) - \sum w_n w_n^* \big) & = & \Psi(z(q)) - \tau(p) \Psi \big(E_{\mathcal{Z}(B)}(\sum_n v_n v_n^*)\big) \\
& = & \Psi \big(E_{\mathcal{Z}(B)}(z(q) - \tau(p) \sum_n v_n v_n^* ) \big) \; .
\end{eqnarray*}
Since $\tau(p) v v^* \leq z(q) - \tau(p) \sum v_n v_n^*$, it is clear that $E_{\mathcal{Z}(B)}(\alpha(v^* v)) \leq E_{\mathcal{Z}(B)}(z(r) - \sum w_n w_n^*)$ and (\ref{eq.subeq}) follows.

We now have that $\alpha(b) = \Psi(b)p_0$ for all $b \in \mathcal{Z}(B)z(q)$. Note that $E_{\mathcal{Z}(B)}(p_0) = \tau(p) z(r)$. Using ergodicity of the action $\Gamma \acts X/i(K)$, one builds a unitary in $u \in M$ such that after composition with $\Ad u$, $\alpha$ satisfies
\begin{equation*}
\alpha(b) = \Psi(b)\tilde{p}  \quad\text{for all}\quad b \in \mathcal{Z}(B) \;,
\end{equation*}
where $\tilde{p}$ is a projection in $B$.

Denote by $(\si_g)_{g \in \Gamma}$ the action of $\Gamma$ on $\rL^\infty(X) \rtimes Z$, implemented by $\Ad u_g$ and corresponding to the Bernoulli action on $\rL^\infty(X)$ and given by $\al_g$ on $\mcL Z$. Since the relative commutant of $\rL^\infty(X/i(K))$ inside $M$ equals $\rL^\infty(X) \rtimes Z$, it follows that
$$\al(u_{\delta(g)}) = \om_g u_g \tilde{p} \quad\text{for all $g \in \Gamma$, where}\quad \om_g \in \tilde{p}(\rL^\infty(X) \rtimes Z)\si_g(\tilde{p}) \; .$$
Note that
$$\om_g \om_g^* = \tilde{p} \;\; , \;\; \om_g^* \om_g = \si_g(\tilde{p}) \quad\text{and}\quad \om_{gh} = \om_g \, \si_g(\om_h) \;\;\text{for all}\;\; g,h \in \Gamma \; .$$
By Theorem \ref{thm.cocycle} (see Section \ref{theorem.cocyclesuperrigidity}) there exists a projection $q \in \B(\ell^2(\N)) \ovt \mcL Z$ with $(\Tr \ot \tau)(q) = \tau(\tilde{p})$, a partial isometry $v \in \B(\C,\ell^2(\N)) \ovt (\rL^\infty(X) \rtimes Z)$ and a family of partial isometries $\gamma_g \in q (\B(\ell^2(\N)) \ovt \mcL Z) \si_g(q)$ satisfying $\gamma_g \gamma_g^* = q, \gamma_g^* \gamma_g = \si_g(q)$ and $\gamma_{gh} = \gamma_g \si_g(\gamma_h)$ such that
$$v^* v = \tilde{p} \;\; , \;\; vv^* = q \quad\text{and}\quad \om_g = v^* \, \gamma_g \, \si_g(v) \;\;\text{for all}\;\; g \in \Gamma \; .$$
Here we view $\B(\C,\ell^2(\N)) \subset \B(\C \oplus \ell^2(\N), \C \oplus \ell^2(\N))$.

Denote by $E$ the map $\Tr \ot \id: \B(\ell^2(\N)) \ovt \mcL Z \recht \mcL Z$. Since $q$ and $\si_g(q)$ are equivalent in $\B(\ell^2(\N)) \ovt \mcL Z$, we get that $E(q) = \si_g(E(q))$ for all $g \in \Gamma$. Remark that $(\mathcal{L} Z,\tau)$ is isomorphic to $(\mathrm{L}^\infty(\mathbb{T}^n),\lambda)$ where $\lambda$ denotes the Lebesgue measure on $\mathbb{T}^n$. By ergodicity of the action $\Gamma \actson \T^n$, it follows that $E(q)$ is equal to a constant $C \neq 0$. Viewing $q$ as a measurable function on $\mathbb{T}^n$ that takes values in the projections of $\B(\ell^2(\mathbb{N}))$, we have that $\Tr(q(x)) = C$ for almost all $x$ in $\mathbb{T}^n$. In particular, $\Tr(q(x)) \geq 1$ almost everywhere. Integrating over $\mathbb{T}^n$, we find that $(\Tr \otimes \tau)(q) \geq 1$. Because $(\Tr \ot \tau)(q) = \tau(\tilde{p}) \leq 1$, we must have $\tilde{p} = 1$. This means that $\alpha$ is an automorphism of $M$.
\end{proof}

\section{$M$ has a Cartan subalgebra that is non-conjugate to $\rL^\infty(X)$}

\begin{theorem}\label{twocart}
With $G \curvearrowright X$ as in Notation \ref{notation}, $\rL^\infty\left(\frac{(\mathbb{Z}_p^n)^\Gamma}{i(\mathbb{Z}_p^n)}\right)\ovt \mathcal{L}\mathbb{Z}^n$ is a Cartan subalgebra of $M=\rL^\infty(X) \rtimes G$ that is non-conjugate to $\rL^\infty(X)$. The associated equivalence relation is the orbit equivalence relation of the action $\widehat{\mathbb{Z}_p^n} \rtimes \Gamma \curvearrowright \mathbb{T}^n \times \frac{(\mathbb{Z}_p^n)^\Gamma}{i(\mathbb{Z}_p^n)}$ where $\widehat{\mathbb{Z}_p^n}$ acts on $\mathbb{T}^n=\widehat{\mathbb{Z}^n}$ by translation and trivially on $\frac{(\mathbb{Z}_p^n)^\Gamma}{i(\mathbb{Z}_p^n)}$, and $\Gamma$ acts on both factors in the natural way.
\end{theorem}

We will prove this theorem using the more general Lemma \ref{Cartan} below.

\begin{lemma}\label{Cartan}
Let $\compZ$ be a compact abelian group and $Z<\compZ$ a countable subgroup. Let $Z_0 < Z$ be an infinite subgroup. Assume that $Z$ acts on $\compZ$ by translation. Let $\Gamma$ be a countable group that acts on $\compZ$ by continuous group automorphisms $(\alpha_g)_{g \in \Gamma}$ preserving $Z$ and $Z_0$. Define $$M:=\mathrm{L}^\infty(\compZ)\rtimes (Z \rtimes \Gamma) \;.$$ Denote $\compZ_0:=\overline{Z_0}$. Assume that for all $g \in \Gamma: \{z- \alpha_g(z) \mid z \in Z_0\}$ is either infinite or trivial.
Denote $\Gamma_0 := \{g \in \Gamma \mid \alpha_g(z) = z\mbox{ for all } z \in Z_0\}$. Assume that $\{x \in \frac{\compZ}{\compZ_0} \mid \alpha_g(x) = x\}$ has infinite index in $\frac{\compZ}{\compZ_0}$ for all $g \in \Gamma_0\backslash\{e\}$.

Assume finally that $Z \cap \compZ_0 = Z_0$. Then $A:=\mathcal{L}(Z_0) \ovt \rL^\infty(\frac{\compZ}{\compZ_0}) =\mathrm{L}^\infty(\widehat{Z_0} \times \frac{\compZ}{\compZ_0})$ is a Cartan subalgebra of $M$ and the induced equivalence relation on $\widehat{Z_0}\times \frac{\compZ}{\compZ_0}$ is given by the action $(\widehat{\compZ_0} \times \frac{Z}{Z_0}) \rtimes \Gamma \curvearrowright \widehat{Z_0}\times \frac{\compZ}{\compZ_0}$, where $\widehat{\compZ_0} \times \frac{Z}{Z_0}$ acts on $\widehat{Z_0} \times \frac{\compZ}{\compZ_0}$ by translation and $\Gamma$ acts on both factors in the natural way.
\end{lemma}

\begin{proof}
Writing the Fourier decomposition of an element in $M$, one easily checks that
\begin{equation}\label{maxab1}
\mathcal{L}(Z_0)'\cap M = \mathrm{L}^\infty(\frac{\compZ}{\compZ_0})\rtimes (Z \rtimes \Gamma_0) \;.
\end{equation}

Now let $x \in \mathrm{L}^\infty(\frac{\compZ}{\compZ_0})' \cap \left(\mathrm{L}^\infty(\frac{\compZ}{\compZ_0})\rtimes(Z \rtimes \Gamma_0)\right)$. Write $x = \sum_{(s,g)\in Z\rtimes \Gamma_0} a_{(s,g)}u_{(s,g)}$ with $a_{s,g} \in \rL^\infty\left(\frac{\compZ}{\compZ_0}\right)$. Then for all $f \in \rL^\infty\left(\frac{\compZ}{\compZ_0}\right)$ we have
$$\sum_{(s,g)\in Z\rtimes\Gamma_0}f a_{(s,g)}u_{(s,g)} =\sum_{(s,g)\in Z\rtimes\Gamma_0} f\left((sZ_0,g)^{-1}\cdot\ \right)a_{(s,g)}u_{(s,g)} \;.$$
So for all $(s,g) \in Z \rtimes \Gamma_0$, we get that $fa_{(s,g)}=f\left((sZ_0,g)^{-1}\cdot\ \right) a_{(s,g)}$ for all $f$. By assumption $\{x \in \frac{\compZ}{\compZ_0}\mid \alpha_g(x) = x\}$ has infinite index in $\frac{\compZ}{\compZ_0}$ for all $g \in \Gamma_0 \setminus \{e\}$. By \cite[Lemma 5]{SV11} we find that $\frac{Z}{Z_0} \rtimes \Gamma_0 \curvearrowright \frac{\compZ}{\compZ_0}$ is essentially free.
So if $(s,g) \not\in Z_0$ then $a_{(s,g)}=0$.
This implies that $$\left(\mathrm{L}^\infty\left(\frac{\compZ}{\compZ_0}\right)\right)' \cap \left(\mathrm{L}^\infty\left(\frac{\compZ}{\compZ_0}\right) \rtimes (Z \rtimes \Gamma_0)\right) = \mathrm{L}^\infty\left(\frac{\compZ}{\compZ_0}\right)\rtimes Z_0 \;.$$
In combination with (\ref{maxab1}), we get that $A$ is maximal abelian in $M$.

For any $\omega \in \hat{\compZ}$ we define the unitary $U_\omega \in \mathrm{L}^\infty(\compZ)$ by $U_\omega(x)=\omega(x)$. One easily checks that $A$ is normalized by $\{u_s \mid s \in Z\}$, $\{u_g \mid g \in \Gamma\}$ and $\{U_\omega\mid \omega \in \widehat{\compZ}\}$.
So $A$ is regular in $M$ and hence $A$ is a Cartan subalgebra of $M$.

It remains to understand the induced equivalence relation. We check how $\Ad u_s$, $\Ad u_g$ and $\Ad U_\omega$ act on $A$ for $s \in Z, g \in \Gamma, \omega \in \widehat{\compZ}$.

It is clear that $\Ad u_s$ does not act on $\mathcal{L}(Z_0)$, but only on $\mathrm{L}^\infty(\frac{\compZ}{\compZ_0})$ by translating by the class of $s$ in $\frac{Z}{Z_0}$. Furthermore $\Ad u_g$ acts both on $\mathcal{L}(Z_0)$ and $\mathrm{L}^\infty(\frac{\compZ}{\compZ_0})$ in the usual way. Finally $\Ad U_\omega$ only acts on $\mathrm{L}^\infty(\widehat{Z_0})$ by restricting to a character on $\compZ_0$ and then translating by this character.

We found actions of $\frac{Z}{Z_0}$, $\Gamma$ and $\widehat{\compZ_0}$ on $\mathrm{L}^\infty(\widehat{Z_0} \times \frac{\compZ}{\compZ_0})$. Remark that for all $s \in Z, g \in \Gamma, \omega \in \widehat{\compZ}$ we have
$$
\Ad u_g \Ad u_s = \Ad u_{\alpha_g(s)} \Ad u_{g}
$$
and
$$
\Ad U_\omega \Ad u_g = \Ad u_g \Ad U_\omega(g^{-1}\cdot \ ) \;.$$
It follows that the induced equivalence relation is given by the orbits of the action $$\left(\widehat{\compZ_0} \times \frac{Z}{Z_0}\right)\rtimes \Gamma \curvearrowright \widehat{Z_0} \times \frac{\compZ}{\compZ_0} \;,$$
where $\widehat{\compZ_0}\times \frac{Z}{Z_0}$ acts on $\widehat{Z_0} \times \frac{\compZ}{\compZ_0}$ by translation and $\Gamma$ acts in the natural way.
\end{proof}

\begin{proof}[Proof of Theorem \ref{twocart}.]
Let $\Gamma$ be as defined in Section 1. We apply Lemma \ref{Cartan} for $\Gamma \curvearrowright \compZ=(\mathbb{Z}_p^n)^\Gamma$ and $Z=Z_0=\mathbb{Z}^n$, where we embed $\mathbb{Z}^n$ into $(\mathbb{Z}_p^n)^\Gamma$ by $i: \mathbb{Z}^n \to (\mathbb{Z}_p^n)^\Gamma: z \mapsto (\alpha_g(z))_{g \in \Gamma}$.

Remark that for any $g \in \Gamma$ the set $\{z- \alpha_g(z) \mid z \in \mathbb{Z}^n\}$ is either infinite or trivial. Indeed, if $x \in \{z - \alpha_g(z) \mid z \in \mathbb{Z}^n\}$ then all integer multiples of $x$ are also in this set. Set $\Gamma_0 := \{g\in \Gamma\mid \alpha_g(z) = z\mbox{ for all } z \in \mathbb{Z}^n\}$. Let $g \in \Gamma_0 \backslash\{e\}$. It is clear that $\{x \in \frac{(\mathbb{Z}_p^n)^\Gamma}{i(\mathbb{Z}_p^n)} \mid \alpha_g(x) = x\}$ has infinite index in $\frac{(\mathbb{Z}_p^n)^\Gamma}{i(\mathbb{Z}_p^n)}$.

So the conditions of Lemma \ref{Cartan} are satisfied and $\rL^\infty\left(\frac{(\mathbb{Z}_p^n)^\Gamma}{i(\mathbb{Z}_p^n)}\right) \ovt \mathcal{L}\mathbb{Z}^n= \mathrm{L}^\infty(\frac{(\mathbb{Z}_p^n)^\Gamma}{i(\mathbb{Z}_p^n)} \times \mathbb{T}^n)$ is a Cartan subalgebra of $M$. The associated equivalence relation is given by the action $\widehat{\mathbb{Z}_p^n} \rtimes \Gamma \curvearrowright \mathbb{T}^n \times \frac{(\mathbb{Z}_p^n)^\Gamma}{i(\mathbb{Z}_p^n)}$ where $\widehat{\mathbb{Z}_p^n}$ only acts on $\mathbb{T}^n=\widehat{\mathbb{Z}^n}$ and $\Gamma$ acts on both factors.
\end{proof}

From now on we denote by $\mathcal{R}_2$ the equivalence relation associated with the Cartan algebra $\rL^\infty(\frac{(\mathbb{Z}_p^n)^\Gamma}{i(\mathbb{Z}_p^n)})\ovt \mathcal{L}(\mathbb{Z}^n)$ of $M$.

\section{The fundamental group of $\mathcal{R}_2$ is non-trivial}

Finally, we prove that the fundamental group of the equivalence relation given by Theorem \ref{twocart} is non-trivial and can be explicitly computed, using techniques from \cite{PV08b}.

The following lemma can easily be checked.
\begin{lemma}\label{SLnZ}
    Let $n \geq 2$. For all $(\mu_1,\ldots,\mu_n)^T \in \mathbb{Z}^n$, we find
    $$\{(z, 0, \ldots, 0)^T, (0,z,\ldots,0)^T,\ldots,(0,0,\ldots,z)^T\} \subset \SLn \cdot (\mu_1, \ldots,\mu_n)^T \;,$$
    where $z = \gcd(\mu_1,\ldots,\mu_n).$
    \end{lemma}

We now describe the $\SLn$-invariant subgroups of $\mathbb{Z}[\frac{1}{p}]^n$.

\begin{lemma}\label{subgroups}
If $G < \mathbb{Z}[\frac{1}{p}]^n$ is a $\SLn$-invariant subgroup then either $G=\{0\}, G=(a\mathbb{Z})^n$ for some $a \in \mathbb{Z}[\frac{1}{p}]$ or $G=(b\mathbb{Z}[\frac{1}{p}])^n$ for some $b \in \mathbb{Z}$.
\end{lemma}
\begin{proof}
Suppose $G \neq \{0\}$. If $G$ is finitely generated, set $G= \langle a_1,\ldots,a_m\rangle$. Choose $l \in \mathbb{N}$ such that
$$a_i=(\frac{a_{i1}}{p^l},\ldots,\frac{a_{in}}{p^l})^T$$
with $a_{i,j} \in \mathbb{Z}$ for all $i,j$. By Lemma \ref{SLnZ}, $G= (\frac{a}{p^l} \mathbb{Z})^n$ with $a= \gcd_{i,j}(a_{ij})$.
Suppose $G$ is not finitely generated. Then, using Lemma \ref{SLnZ}, for each $k \in \mathbb{N}$ we find $m_k \geq k$ and $a_k \in \mathbb{Z}, p\nmid a_k$ such that $(\frac{a_k}{p^{m_k}},\ldots,\frac{a_k}{p^{m_k}})^T \in G$ and if $|a'| < |a_k|$ then $(\frac{a'}{p^{m_k}},\ldots,\frac{a'}{p^{m_k}})^T \not\in G$. Set $b_k = \gcd_{i\leq k}(a_i)$. Let $b= \lim_{k \to \infty} b_k$. One verifies that $G = (b\mathbb{Z}[\frac{1}{p}])^n$.
\end{proof}

\begin{theorem}\label{nontriv}
Let $G \curvearrowright X$ as in Notation \ref{notation}. Denote by $\mathcal{R}_2$ the equivalence relation associated with the Cartan subalgebra $\rL^\infty(\frac{(\mathbb{Z}_p^n)^\Gamma}{i(\mathbb{Z}_p^n)})\ovt \mathcal{L}(\mathbb{Z}^n)$ of $M$. Then $\mathcal{F}(\mathcal{R}_2) = \{p^{kn}\mid k \in \mathbb{Z}\}$.
\end{theorem}

\begin{proof}
It follows from Theorem \ref{twocart} that $\mathcal{R}_2$ is the orbit equivalence relation of the action $\widehat{\mathbb{Z}_p^n} \rtimes \Gamma \curvearrowright \mathbb{T}^n \times \frac{(\mathbb{Z}_p^n)^\Gamma}{i(\mathbb{Z}_p^n)}$. We will compute its fundamental group using \cite[Lemma 5.10]{PV08b}. Therefore we write the action as a quotient action $\frac{\tilde{G}}{N} \curvearrowright \frac{\tilde{X}}{N}$.

Recall that $q:\Gamma \to \SLn$ denotes the quotient map. Let $\Gamma$ act on $\mathbb{Z}[\frac{1}{p}]^n$ and on $\mathbb{Z}_p^n$ through $q$. Denote
$$\tilde{G} = \left(\mathbb{Z}[\frac{1}{p}]^n \times \mathbb{Z}_p^n\right) \rtimes \Gamma\quad \mbox{and}\quad \tilde{X}= \mathbb{R}^n \times (\mathbb{Z}_p^n)^\Gamma \;.$$
Define $i:\mathbb{Z}_p^n \to (\mathbb{Z}_p^n)^\Gamma:z \mapsto (g\cdot z)_{g \in \Gamma}$. Let $\tilde{G} \curvearrowright \tilde{X}$ where $\mathbb{Z}[\frac{1}{p}]^n \times \mathbb{Z}_p^n$ acts by translation after embedding by $i$ and $\Gamma$ acts on $\mathbb{R}^n$ through $q$ and on $(\mathbb{Z}_p^n)^\Gamma$ by Bernoulli shift.
Set $N = \mathbb{Z}^n \times \mathbb{Z}_p^n$. Then $N \lhd \tilde{G}$ is an open normal subgroup and the restricted action $N \curvearrowright \tilde{X}$ is proper. Observe that $\left(\frac{\tilde{G}}{N} \curvearrowright \frac{\tilde{X}}{N}\right) = \left(\widehat{\mathbb{Z}}_p^n \rtimes \Gamma \curvearrowright \mathbb{T}^n \times \frac{(\mathbb{Z}_p^n)^\Gamma}{i(\mathbb{Z}_p^n)}\right)$.
We will prove Theorem \ref{nontriv} in two steps.

\textbf{Step 1:} $\tilde{G} \curvearrowright \tilde{X}$ is $\mathcal{U}_{fin}$-cocycle superrigid.

We first prove that $\SLn \curvearrowright \tilde{X}$ is $\mathcal{U}_{fin}$-cocycle superrigid, using \cite[Theorem 5.3]{PV08b}. So we need to prove that $\SLn \curvearrowright \tilde{X}$ is s-malleable, that the diagonal action $\SLn \curvearrowright \tilde{X}^2$ has property (T) and that the 4-fold diagonal action $\SLn \curvearrowright \tilde{X}^4$ is ergodic.

As $\SLn \curvearrowright (\mathbb{Z}_p^n)^\Gamma$ is a Bernoulli action with diffuse base space, it is s-malleable (see \cite[\S 1.6]{Po03}). Also $\SLn \curvearrowright \mathbb{R}^n$
is s-malleable (see the proof of \cite[Theorem 1.3]{PV08b}), hence so is $\SLn \curvearrowright \tilde{X}$.

Next we prove that $\SLn \curvearrowright \tilde{X}^2$ has property (T). This is very similar to the proof of \cite[Lemma 5.6]{PV08b}. Denote by $(e_i)_{i=1,\ldots,n}$ the standard basis vectors in $\mathbb{R}^n$. The orbit of $(e_1,e_2)\in \mathbb{R}^n\times \mathbb{R}^n$ under the diagonal $\SLnR$-action has complement of measure zero. Hence we can identify $\SLn \curvearrowright \tilde{X}^2$ with $\SLn \curvearrowright \SLnR/ H_0 \times X \times X$ where $H_0= \mathrm{Stab}_{\SLnR}(e_1,e_2)$. Now by \cite[Proposition 3.5]{PV08b} $\SLn \curvearrowright \tilde{X}^2$ has property (T) if and only if $H_0 \curvearrowright \SLn \backslash \SLnR \times X \times X$ has property (T). But $\SLn \backslash \SLnR \times X \times X$ is a probability space, and as $H_0$ has property (T), so does the action $H_0 \curvearrowright \SLn \backslash \SLnR \times X \times X$ by \cite[Proposition 3.2]{PV08b}.

We still need to show that $\SLn \curvearrowright \tilde{X}^4$ is ergodic. From \cite[Lemma 5.6]{PV08b} we know that $\SLn \curvearrowright (\mathbb{R}^n)^4$ is ergodic. In fact this action is properly ergodic. Now $\SLn \curvearrowright (\mathbb{Z}_p^n)^\Gamma$ is mixing and hence mildly mixing. It follows from Proposition \ref{mildly} that $\SLn \curvearrowright \tilde{X}^4$ is ergodic.

So $\SLn \curvearrowright \mathbb{R}^n \times (\mathbb{Z}_p^n)^\Gamma$ is $\mathcal{U}_{fin}$-cocycle superrigid.

Now let $\omega: \left((\mathbb{Z}[\frac{1}{p}] \times \mathbb{Z}_p)^n\rtimes \Gamma\right) \times (\mathbb{R}^n \times (\mathbb{Z}_p^n)^\Gamma) \to \mathcal{G}$ be a $1$-cocycle for the action $\tilde{G}\curvearrowright \tilde{X}$ with values in a Polish group of finite type. By the previous paragraphs, we may assume that $\omega_{\mid\SLn}$ is a group morphism. Let $a \in \mathbb{Z}[\frac{1}{p}]$, $b \in \mathbb{Z}_p$ be any two elements. Write
\begin{equation*}
(a,b)_i = \big(((0,0),\ldots,(0,0),(a,b),(0,0),\ldots,(0,0))^T,e \big)\in (\mathbb{Z}[\frac{1}{p}] \times \mathbb{Z}_p)^n \rtimes \Gamma
\end{equation*}
where we write $(a,b)$ in the $i^{th}$ row. Then $\SLn \cap (a,b)_i \SLn (a,b)_i^{-1} = H_i$, where $H_i$ is the group of matrices in $\SLn$ that leave all vectors $(a,b)_i$ invariant. One easily sees that this means the $i^{th}$ column of the matrix has to be the $i^{th}$ unit vector. Note that $H_i$ is isomorphic to $\mathrm{SL}_{n-1}(\mathbb{Z}) \ltimes \mathbb{Z}^{n-1}$.

By Lemma \ref{cocycle} it now suffices to prove that $H_i \curvearrowright \tilde{X}^2$ is ergodic for all $i$ to get that $\omega$ is a group morphism on $(\mathbb{Z}[\frac{1}{p}]^n \times \mathbb{Z}_p^n) \rtimes \mathrm{SL}_{n}(\mathbb{Z})$. But as $H_i$ acts mixingly on $(\mathbb{Z}_p^n)^\Gamma$, by Proposition \ref{mildly} it suffices to prove that $H_i \curvearrowright (\mathbb{R}^n)^2$ is properly ergodic.

We show that $H_1 \curvearrowright (\mathbb{R}^n)^2$ is properly ergodic. Let $F: \mathbb{R}^n \times \mathbb{R}^n \to \mathbb{R}$ be an $H_1$-invariant function. For all $x_1, x_{n+1} \in \mathbb{R}$, the map $F_{x_1,x_{n+1}}: \mathbb{R}^{n-1} \times \mathbb{R}^{n-1} \to \mathbb{R}: (x_2,\ldots,x_n, x_{n+2},\ldots,x_{2n}) \mapsto F(x_1,\ldots,x_{2n})$ is $\mathrm{SL}_{n-1}(\mathbb{Z})$-invariant. By ergodicity of the diagonal action $\mathrm{SL}_{n-1}(\mathbb{Z}) \curvearrowright (\mathbb{R}^{n-1})^2$ (see \cite[Lemma 5.6]{PV08b}) we find that $F_{x_1, x_{n+1}}$ is essentially constant for all $x_1, x_{n+1} \in \mathbb{R}$, say $c_{x_1,x_{n+1}}$. Set $E: \mathbb{R}^2 \to \mathbb{R}: (x_1,x_{n+1}) \mapsto c_{x_1, x_{n+1}}$. Then for almost all $(x_1,x_{n+1}) \in \mathbb{R}^2$, for almost all $(x_2,\ldots,x_n, x_{n+2},\ldots,x_{2n}) \in \mathbb{R}^{n-1}\times \mathbb{R}^{n-1}$ and for all $a_2,\ldots,a_n \in \mathbb{Z}$ we have
\begin{eqnarray*}
& & E(x_1,x_{n+1})\\ & = & F(x_1,\ldots,x_{2n})\\
 & = & F(x_1 +a_2 x_2 +\ldots+a_nx_n,x_2,\ldots,x_n, x_{n+1} + a_{2}x_{n+2} + \ldots + a_{n}x_{2n}, x_{n+2},\ldots,x_{2n})\\
 & = & E(x_1+a_2 x_2 +\ldots+a_nx_n, x_{n+1} + a_{2}x_{n+2} + \ldots + a_{n}x_{2n}) \;.
\end{eqnarray*}
Hence $E$ is essentially constant. It follows that $F$ is essentially constant. Furthermore it is clear that any $H_1$-orbit is negligible. So $H_1 \curvearrowright (\mathbb{R}^n)^2$ is properly ergodic. The same reasoning holds for all $i$, so that $\omega$ is a group morphism on $(\mathbb{Z}[\frac{1}{p}]^n \times \mathbb{Z}_p^n) \rtimes \SLn$.

In particular $\omega$ is a group morphism on $(\mathbb{Z}[\frac{1}{p}]^n \times \mathbb{Z}_p^n) \rtimes \Sigma$ and this commutes with $\Lambda$. To conclude that $(\mathbb{Z}[\frac{1}{p}]^n \times \mathbb{Z}_p^n) \rtimes \Gamma \curvearrowright \mathbb{R}^n \times (\mathbb{Z}_p^n)^\Gamma$ is $\mathcal{U}_{fin}$-cocycle superrigid, by Lemma \ref{cocycle} it suffices to prove that the diagonal action $(\mathbb{Z}[\frac{1}{p}]^n \times \mathbb{Z}_p^n) \rtimes \Sigma \curvearrowright \mathbb{R}^n \times \mathbb{R}^n \times ((\mathbb{Z}_p^n)^\Gamma)^{2}$ is ergodic. Let $F: \mathbb{R}^n \times \mathbb{R}^n \times ((\mathbb{Z}_p^n)^\Gamma)^{2} \to \mathbb{R}$ be a $(\mathbb{Z}[\frac{1}{p}]^n \times \mathbb{Z}_p^n) \rtimes \Sigma$-invariant function. As $\mathbb{Z}[\frac{1}{p}]^n$ is dense in $\mathbb{R}^n$, $F$ satisfies $F(x,y,z) = H(x-y,z)$ where $H: \mathbb{R}^n \times ((\mathbb{Z}_p^n)^\Gamma)^{2} \to \mathbb{R}$ is $\mathbb{Z}_p^n \rtimes \Sigma$-invariant.

So $H$ is $\Sigma$-invariant. As $\Sigma$ acts on $((\mathbb{Z}_p^n)^\Gamma)^2$ mildly mixingly, it suffices to see that $\Sigma \curvearrowright \mathbb{R}^n$ is ergodic. But this is true, as $\mathrm{SL}_2(\mathbb{Z}) \curvearrowright \mathbb{R}^2$ is ergodic.

So we have shown that $(\mathbb{Z}[\frac{1}{p}]^n \times \mathbb{Z}_p^n)\rtimes \Gamma \curvearrowright \mathbb{R}^n \times (\mathbb{Z}_p^n)^\Gamma$ is $\mathcal{U}_{fin}$-cocycle superrigid.

\textbf{Step 2:} $\mathcal{F}(\mathcal{R}_2) = \{ p^{kn} \mid k \in \mathbb{Z} \}$.

To prove this we use \cite[Lemma 5.10]{PV08b}.
Let $\tilde{\mu}$ be the canonical measure on $\tilde{X}$ and fix a Haar measure $\lambda$ on $\tilde{G}$. Remark that the action $\tilde{G} \curvearrowright (\tilde{X},\tilde{\mu})$ is essentially free, ergodic and measure preserving. 

Note that for any open normal subgroup $N_0$ of $\tilde{G}$ for which the restricted action $N_0 \actson \tilde{X}$ is essentially free and proper, there exists a canonical (not necessarily finite) measure $\mu_{N_0}$ on $\frac{\tilde{X}}{N_0}$ such that $\frac{\tilde{G}}{N_0} \curvearrowright (\frac{\tilde{X}}{N_0}, \mu_{N_0})$ is ergodic and measure preserving. Indeed, since $N_0$ acts essentially freely and properly on $\tilde{X}$, we can write  $\tilde{X}$ as $N_0 \times \frac{\tilde{X}}{N_0}$ and under this bijection $\tilde{\mu}$ corresponds to $\lambda_{|N_0} \times \mu_{N_0}$ . In the rest of the proof we will make use of these (not necessarily normalized) measures.

Recall that $N=\mathbb{Z}^n \times \mathbb{Z}_p^n$ is an open normal subgroup of $\tilde{G}$ and that the restricted action $N \actson \tilde{X}$ is essentially free and proper. Remark that $\mu_N$ is a finite measure.

Let $\Delta: \frac{\tilde{X}}{N} \to \frac{\tilde{X}}{N}$ be a stable orbit equivalence between
$\frac{\tilde{G}}{N} \stackrel{\alpha}{\curvearrowright} \frac{\tilde{X}}{N}$ and itself. By \cite[Lemma 5.10]{PV08b}, there exists
\begin{itemize}
\item a subgroup $\Lambda_0 < \frac{\tilde{G}}{N}$ and a non-negligible subset $Y_0 \subset \frac{\tilde{X}}{N}$, such that $\alpha$ is induced from $\Lambda_0 \curvearrowright Y_0$,
\item an open normal subgroup $N_1 \lhd \tilde{G}$ such that the restricted action $N_1 \curvearrowright \tilde{X}$ is proper,
\end{itemize}
such that $\frac{\tilde{G}}{N_1} \curvearrowright \frac{\tilde{X}}{N_1}$ and $\Lambda_0 \curvearrowright Y_0$ are conjugate through a non-singular isomorphism and a group isomorphism. But $\frac{\tilde{G}}{N} \curvearrowright \frac{\tilde{X}}{N}$ is weakly mixing and hence it is not an induced action.

So we find $N_1 \lhd \tilde{G}$ open, such that $N_1 \curvearrowright \tilde{X}$ is proper and such that $\frac{\tilde{G}}{N_1} \curvearrowright \frac{\tilde{X}}{N_1}$ and $\frac{\tilde{G}}{N}\curvearrowright \frac{\tilde{X}}{N}$ are conjugate through the non-singular isomorphism $\Psi: \frac{\tilde{X}}{N_1} \to \frac{\tilde{X}}{N}$ and the group isomorphism $\delta: \frac{\tilde{G}}{N_1} \to \frac{\tilde{G}}{N}$. Furthermore we have 
\begin{equation}\label{comprconst}
\Delta(N\cdot x) \in \frac{\tilde{G}}{N} \cdot \Psi(N_1\cdot x) \mbox{ for almost all } x \in \tilde{X} \; .
\end{equation}
It follows from the ergodicity of the actions $\frac{\tilde{G}}{N_1} \curvearrowright \frac{\tilde{X}}{N_1}$ and $\frac{\tilde{G}}{N}\curvearrowright \frac{\tilde{X}}{N}$ that $\Psi$ is measure scaling and that $\mu_{N_1}$ is a finite measure. The compression constant $c(\Psi)$ equals $\frac{\mu_N(\tilde{X}/N)}{\mu_{N_1}(\tilde{X}/N_1)}.$ 
By (\ref{comprconst}) the compression constant of $\Delta$ is the product of $c(\Psi)$ and the compression constant of the canonical stable orbit equivalence between $\frac{\tilde{G}}{N}\curvearrowright \frac{\tilde{X}}{N}$ and $\frac{\tilde{G}}{N_1}\curvearrowright \frac{\tilde{X}}{N_1}$. The latter is $1$ (with respect to the non-normalized measures $\mu_N$ and $\mu_{N_1}$). Hence $c(\Delta)=c(\Psi) = \frac{\mu_N(\tilde{X}/N)}{\mu_{N_1}(\tilde{X}/N_1)}$.

Set $P:=(\mathbb{Z}[\frac{1}{p}]\times \mathbb{Z}_p)^n$. We show that $N_1 \lhd P$. Let $q: \Gamma \to \SLn$ be the quotient map as in Notation \ref{notation} and write $\Gamma_0 = \ker(q)$. Then $\Gamma_0 \curvearrowright \tilde{X}$ is given by $\mathrm{id} \times \Gamma_0 \curvearrowright \mathbb{R}^n \times (\mathbb{Z}_p^n)^\Gamma$ where $\Gamma_0$ acts on $(\mathbb{Z}_p^n)^\Gamma$ by Bernoulli shift. Now consider the subset $[0,1]^n\times (\mathbb{Z}_p^n)^\Gamma \subset \tilde{X}$. This set is globally invariant under $N_1\cap \Gamma_0$ and $N_1\cap \Gamma_0 \curvearrowright [0,1]^n\times (\mathbb{Z}_p^n)^\Gamma$ is still essentially free and proper. So there exists a measurable map $\pi: [0,1]^n\times (\mathbb{Z}_p^n)^\Gamma \rightarrow N_1\cap \Gamma_0$ such that $\pi(g \cdot x) = g \pi(x)$ for almost all $(g,x) \in N_1\cap \Gamma_0 \times ([0,1]^n\times (\mathbb{Z}_p^n)^\Gamma)$. But this implies that $N_1 \cap \Gamma_0$ is compact, as the pushforward of the invariant probability measure on $[0,1]^n\
times (\mathbb{Z}_p^n)^\Gamma$ is an invariant probability measure on $N_1 \cap \Gamma_0$. So $N_1 \cap \Gamma_0$ is finite. Being an icc group, $\Gamma$ has no non-trivial finite normal subgroups. As $N_1 \cap \Gamma_0$ is normal in $\Gamma$, it has to be trivial.

Now take $xg \in N_1$ where $x \in P, g \in \Gamma$. Then for all $\lambda \in \Gamma_0$ we have
$$N_1 \ni (xg)^{-1} \lambda (xg) \lambda^{-1} = g^{-1}x^{-1}\lambda x g \lambda^{-1} = g^{-1} \lambda g \lambda^{-1} \in \Gamma_0 \;.$$
So $\lambda g \lambda^{-1}=g$ for all $\lambda \in \Gamma_0$, hence $g=e$.
We have found that $N_1 \lhd P$.

As $N_1$ is open, there exists $l \in \mathbb{N}$ such that $(\{0\}  \times p^l\mathbb{Z}_p)^n < N_1$. Since $(\mathbb{Z}[\frac{1}{p}] \times p^l \mathbb{Z}_p)^n < P$ has finite index, also $N_2:= N_1 \cap (\mathbb{Z}[\frac{1}{p}] \times p^l \mathbb{Z}_p)^n < N_1$ has finite index. It is clear that $N_2 = N_3 \times (p^l \mathbb{Z}_p)^n$ for some $N_3 < \mathbb{Z}[\frac{1}{p}]^n$. Furthermore $N_3$ is $\SLn$-invariant as $N_2$ is normal in $\tilde{G}$.

By Lemma \ref{subgroups}, $N_3$ is either $\{0\}, (a\mathbb{Z})^n$ for some $a \in \mathbb{Z}[\frac{1}{p}]$ or $(b\mathbb{Z}[\frac{1}{p}])^n$ for some $b \in \mathbb{Z}$. Remark that $N_2 \curvearrowright \tilde{X}$ is proper. Hence $N_3 \neq (b\mathbb{Z}[\frac{1}{p}])^n$ for $b \in \mathbb{Z}$. As $N_2$ has finite index in $N_1$ and $N_1 \curvearrowright \tilde{X}$ has finite covolume, also $N_2 \curvearrowright \tilde{X}$ has finite covolume. So $N_3 \neq\{0\}$. We conclude that $N_3=(a\mathbb{Z})^n$ for some $a \in \mathbb{Z}[\frac{1}{p}]$ and $N_2 = (a\mathbb{Z} \times p^l \mathbb{Z}_p)^n$.

In particular we found $N_2 < N_1$ such that $N_2$ and $N=\mathbb{Z}^n \times \mathbb{Z}_p^n$ are commensurate. But as $N_2$ has finite index in $N_1$, this implies that $N_1$ and $N$ are commensurate. It follows that $\frac{\mu_N(\tilde{X}/N)}{\mu_{N_1}(\tilde{X}/N_1)} = \frac{[N_1 : N \cap N_1]}{[N : N \cap N_1]}$. So $c(\Delta) = \frac{[N_1 : N \cap N_1]}{[N : N \cap N_1]}$.

Set \begin{equation*}
p_1 : N \hookrightarrow P \twoheadrightarrow \frac{P}{N_1} \hookrightarrow \frac{\tilde{G}}{N_1} \cong \frac{\tilde{G}}{N}  \quad \;\text{and} \; \quad
p_2 : N_1 \hookrightarrow P \twoheadrightarrow \frac{P}{N} \hookrightarrow \frac{\tilde{G}}{N} \;.
\end{equation*}
As $N$ and $N_1$ are commensurate, both $\operatorname{Im}(p_1)$ and $\operatorname{Im}(p_2)$ are finite abelian normal subgroups of $\frac{\tilde{G}}{N}$. But then $p_\Gamma(\operatorname{Im}(p_i))=\{e\}$ for $i=1,2$, so $\operatorname{Im}(p_i)$ is a finite subgroup of a $p$-group. One verifies that $|\operatorname{Im}(p_i)|$ is a power of $p^n$, say $|\operatorname{Im}(p_i)|=p^{nl_i}$ for $i=1,2$, $l_i \in \mathbb{Z}$.
So we see that
\begin{eqnarray*}
p^{nl_1}&=&|\operatorname{Im}(p_1)|=[N : \ker(p_1)] = [N : N \cap N_1] \;,\\
p^{nl_2}&=&|\operatorname{Im}(p_2)|=[N_1 : \ker(p_2)] = [N_1 : N \cap N_1] \;.
\end{eqnarray*}

It follows that $c(\Delta) = \frac{[N_1 : N \cap N_1]}{[N : N \cap N_1]} = p^{kn}$ for some $k \in \mathbb{Z}$. Since $\mathcal{F}(\mathcal{R}_2)$ is the group generated by the compression constants of stable orbit equivalences between $\mathcal{R}_2$ and itself, we find that $\mathcal{F}(\mathcal{R}_2) \subset \{ p^{kn} \mid k \in \mathbb{Z} \}$.

To prove the other inclusion remark that $N_1=\frac{1}{p}\mathbb{Z}^n \times \mathbb{Z}_p^n \lhd \mathbb{Z}[\frac{1}{p}]^n \times \mathbb{Z}_p^n$ is open and that the restricted action $N_1 \curvearrowright \tilde{X}$ is proper. Define the group isomorphism $\delta: \frac{\tilde{G}}{N_1} \to \frac{\tilde{G}}{N}$ by $\delta(z,s,g) = (pz,s,g)$. Set $\Psi: \frac{\tilde{X}}{N_1} \to \frac{\tilde{X}}{N}: \Psi(x,y) = (px,y)$. It is clear that this defines a measure preserving isomorphism. Then composition of $\Psi$ with the canonical stable orbit equivalence between $\frac{\tilde{G}}{N} \curvearrowright \frac{\tilde{X}}{N}$ and $\frac{\tilde{G}}{N_1} \curvearrowright \frac{\tilde{X}}{N_1}$ defines a self stable orbit equivalence of $\frac{\tilde{G}}{N} \curvearrowright \frac{\tilde{X}}{N}$ with compression constant $p^n$. So $\mathcal{F}(\mathcal{R}_2) = \{ p^{kn} \mid k \in \mathbb{Z} \}$.
\end{proof}

\section{A cocycle superrigidity theorem}\label{theorem.cocyclesuperrigidity}

We prove the following twisted version of Popa's cocycle superrigidity theorem \cite[Theorem 5.5]{Po05} for Bernoulli actions of groups admitting an infinite rigid subgroup that is wq-normal (see the discussion preceding Lemma \ref{cocyclesuperrigid} for the terminology).

The theorem is an adapted version of \cite[Theorem 11]{SV11} for generalized 1-cocycles.

\begin{definition}
Let $(\si_g)_{g \in \Gamma}$ be an action of a countable group $\Gamma$ on a von Neumann algebra $N$. Let $q \in N$ be a projection. A generalized 1-cocycle for the action $(\si_g)_{g \in \Gamma}$ on $N$ with support $q$ is a family of partial isometries $(\gamma_g)_{g \in \Gamma}$ in $q N \si_g(q)$, satisfying $\gamma_g \gamma_g^* = q, \gamma_g^* \gamma_g = \si_g(q)$ and $\gamma_{gh} = \gamma_g \, \si_g(\gamma_h)$ for all $g,h \in \Gamma$.
\end{definition}

For any von Neumann algebra $M$ we denote by $M^\infty$ the von Neumann algebra $M^\infty:=\B(\ell^2(\N)) \ovt M$.

\begin{theorem}\label{thm.cocycle}
Let $K$ be a compact group with countable subgroup $Z < K$. Let $\Gamma$ be a countable group. Assume that $\Gamma$ admits an infinite rigid subgroup that is wq-normal in $\Gamma$. Put $X = K^\Gamma$. Assume that $\Gamma$ acts on $K$ by continuous group automorphisms $(\al_g)_{g \in \Gamma}$ preserving $Z$. Embed $Z$ in $X$ by $i: Z \rightarrow X: z \mapsto (\alpha_g(z))_{g \in \Gamma}$. Put
$$N := \rL^\infty(X) \rtimes Z$$
where $Z$ acts on $X$ by translation after embedding by $i$. Denote by $(\si_g)_{g \in \Gamma}$ the action of $\Gamma$ on $N$ such that $\si_g$ is the Bernoulli shift on $\rL^\infty(X)$ and is given by $\alpha_g$ on $\mcL Z$. Let $p \in N$ be a non-zero projection.

\begin{itemize}
\item Assume that $q \in (\mcL Z)^\infty$ is a projection and that $(\gamma_g)_{g \in \Gamma}$ is a generalized 1-cocycle for the action of $\Gamma$ on $(\mcL Z)^\infty$ with support $q$. Assume that $v \in \B(\C,\ell^2(\N)) \ovt N$ is a partial isometry satisfying $v^* v = p$ and $v v^* = q$. Then the formula
\begin{equation}\label{eq.gen-cocycle}
    \om_g := v^* \, \gamma_g \, \si_g(v)
\end{equation}
defines a generalized 1-cocycle $(\omega_g)_{g \in \Gamma}$ for the action $(\si_g)_{g \in \Gamma}$ on $N$ with support $p$.

\item Conversely, every generalized $1$-cocycle for the action $(\si_g)_{g \in \Gamma}$ on $N$ is of the above form with $\gamma$ being uniquely determined in the following sense: if $(\gamma_g)_g$ and $(\vphi_g)_g$ both satisfy (\ref{eq.gen-cocycle}), then there is a unitary $u \in (\mcL Z)^\infty$ such that $\vphi_g = u \gamma_g \sigma_g(u^*)$ for all $g \in \Gamma$.
\end{itemize}
\end{theorem}

\begin{proof}
It is clear that the formulae in the theorem define generalized $1$-cocycles. Conversely, let $p \in N$ be a projection and assume that the partial isometries $(\om_g)_{g \in \Gamma}$ define a generalized $1$-cocycle for the action $(\si_g)_{g \in \Gamma}$ on $N$ with support $p$. Denote by $E_{\mathcal{Z}(N)} : \rL^\infty(X) \rtimes Z \rightarrow \rL^\infty(X/\overline{i(Z)})$ the trace preserving conditional expectation of $N$ onto its center. Since $p$ and $\si_g(p)$ are equivalent in $N$, we have $E_{\mathcal{Z}(N)}(p) = E_{\mathcal{Z}(N)}(\si_g(p)) = \si_g(E_{\mathcal{Z}(N)}(p))$ for all $g \in \Gamma$. By ergodicity of $\Gamma \actson X/\overline{i(Z)}$, it follows that $E_{\mathcal{Z}(N)}(p)$ is constant, i.e. $E_{\mathcal{Z}(N)}(p) = \tau(p)$ where $\tau$ denotes the tracial state on $N$. But then $p$ is equivalent in $N$ to any projection in $\mcL Z$ of trace $\tau(p)$. So we may assume that $p \in \mcL Z$.

Consider the action $Z \actson X \times K$ given by
$$z \cdot ((x_g)_g,k) = ((\al_g(z) x_g)_g, z k) \quad\text{for all}\;\; z \in Z, (x_g)_g \in X, k \in K \; .$$
Put $\cN := \rL^\infty(X \times K) \rtimes Z$. We embed $N \subset \cN$ by identifying the element $F u_z \in N$ with the element $(F \ot 1)u_z \in \cN$ whenever $F \in \rL^\infty(X), z \in Z$. Also $(\si_g)_{g \in \Gamma}$ extends naturally to a group of automorphisms of $\cN$ with $\si_g(1 \ot F) = 1 \ot (F \circ \alpha_{g^{-1}})$ for all $F \in \rL^\infty(K)$. Define $P = \rL^\infty(K) \rtimes Z$. View $P$ as a subalgebra of $\cN$ by identifying the element $Fu_z \in P$ with the element $(1 \ot F)u_z \in \cN$ whenever $F \in \rL^\infty(K), z \in Z$. Remark that the restriction of $\si_g$ to $P$ is given by $\si_g(Fu_z) = (F \circ \al_{g^{-1}}) u_{\al_g(z)}$ for all $F \in \rL^\infty(K), z \in Z$.

Denote by $R$ the hyperfinite II$_1$ factor. Let $r \in R$ be a projection of trace $\tau(p)$. Recall that $p \in \mcL Z$. Denoting by $E_{\mathcal{Z}(P \ovt R)}: P \ovt R \rightarrow \rL^\infty(K/\overline{i(Z)}) \ot 1$ the trace preserving conditional expectation of $P \ovt R$ onto its center, we have that $E_{\mathcal{Z}(P \ovt R)}(1 \ot r) = E_{\mathcal{Z}(P \ovt R)}(p \ot 1) = \tau(p)$. It follows that $1 \ot r$ and $p \ot 1$ are equivalent in $P \ovt R$. Take $u \in P \ovt R$ with $uu^* = 1 \ot r$ and $u^*u = p \ot 1$. Now view $u$ and $\omega_g$ in $\cN \ovt R$ and define
$$\nu_g := u \omega_g (\si_g \ot \id)(u^*) \in \cU(\cN \ovt rRr) \;.$$
One verifies that $\nu_g$ is a $1$-cocycle for the action $(\si_g \ot \id)_{g \in \Gamma}$ on $\cN \ovt rRr$, i.e.\ a family of unitaries satisfying $\nu_{gh} = \nu_g \, (\si_g \ot \id) (\nu_h)$ for all $g,h \in \Gamma$.

Define $\Delta : X \times K \recht X \times K$ by $\Delta((x_g)_g,k) = ((\al_g(k) x_g)_g, k)$ and denote by $\Delta_*$ the automorphism of $\rL^\infty(X \times K)$ given by $\Delta_*(F) = F \circ \Delta^{-1}$. One checks that the formula
$$\Phi : \rL^\infty(X) \ovt P \recht \cN : \Phi(F \ot G u_z) = \Delta_*(F \ot G) u_z$$
for all $F \in \rL^\infty(X) , G \in \rL^\infty(K)$ and $z \in Z,$ defines a $*$-isomorphism satisfying $\Phi \circ (\si_g \ot \si_g) = \si_g \circ \Phi$ for all $g \in \Gamma$. Define
$$\Psi = \Phi \ot \id_R: \rL^\infty(X) \ovt P \ovt R \recht \cN \ovt R \: .$$
Put $\mu_g := \Psi^{-1}(\nu_g)$. It follows that $(\mu_g)_{g \in \Gamma}$ is a 1-cocycle for the action $(\si_g \ot \si_g \ot \id)_{g \in \Gamma}$ on $\rL^\infty(X) \ovt P \ovt rRr$. By Popa's cocycle superrigidity theorem \cite[Theorem 5.5]{Po05} and directly applying $\Psi$ again, we find a unitary $v \in \cU(\cN \ovt rRr)$ and a 1-cocycle $\delta_g \in \cU(P \ovt rRr)$ for the action $(\si_g \ot \id)_{g \in \Gamma}$ on $P \ovt rRr$ such that $\nu_g = v^* \delta_g (\si_g \ot \id)(v)$ for all $g \in \Gamma$.

Define $w := u^* v u$ and $\rho_g := u^* \delta_g (\si_g \ot \id)(u)$ for all $g \in \Gamma$. Then $w \in \cU(p\cN p \ovt R)$. Furthermore $\rho_g \in pP\si_g(p) \ovt R$ is a family of partial isometries satisfying $\rho_g \rho_g^* = p \ot 1, \rho_g^* \rho_g = \si_g(p) \ot 1$ and $\rho_{gh} = \rho_g (\si_g \ot \id)(\rho_h)$ such that
$$\omega_g = w^* \, \rho_g \, (\si_g \ot \id)(w) \quad\text{for all}\;\; g \in \Gamma \; .$$

Consider the basic construction for the inclusion $N \subset \cN \ovt R$ denoted by $\cN_1 := \langle \cN \ovt R, e_N \rangle$. Put $T := w e_N w^*$. Since $\si_g(w) = \rho_g^* \, w \, \om_g$, it follows that $\si_g(T) = \rho_g^* \, T \, \rho_g$. Also note that $T \in \rL^2(\cN_1)$, that $T = p T$ and that $\Tr(T) = \tau(p)$.

Define for every finite subset $\cF \subset \Gamma$, the von Neumann subalgebras $N_\cF \subset N$ and $\cN_\cF \subset \cN$ given by
$$N_\cF := \rL^\infty(K^\cF) \rtimes Z \;\;\text{and}\;\; \cN_\cF := \rL^\infty(K^\cF \times K) \rtimes Z \; .$$
For every finite subset $\cF \subset \Gamma$, we have the following commuting square
$$
\begin{matrix}
N & \subset & \cN \ovt R\;\;\;\mbox{} \\
\cup & & \cup\;\;\;\mbox{} \\
N_\cF & \subset & \cN_\cF \ovt R \;\;.
\end{matrix}
$$
Remark furthermore that $N (\cN_\cF \ovt R)$ is dense in $\cN \ovt R$. It follows that we can identify the basic construction $\langle \cN_\cF \ovt R, e_{N_\cF} \rangle$ for the inclusion $N_\cF \subset \cN_\cF \ovt R$ with the von Neumann subalgebra of $\cN_1$ generated by $\cN_\cF \ovt R$ and $e_N$. Remark that for $\cF = \emptyset$, we get that the basic construction $P_1 := \langle P \ovt R, e_{\mcL Z} \rangle$ is isomorphic to the von Neumann subalgebra of $\cN_1$ generated by $P \ovt R$ and $e_N$. Denote by $\|\,\cdot\,\|_2$ the $2$-norm on $\rL^2(\cN_1)$ given by the semi-finite trace. Under this identification we have
$$ \overline{\bigcup_{\cF \subset \Gamma} \rL^2(\langle \cN_\cF \ovt R, e_{N_\cF} \rangle)}^{\|\,\cdot\,\|_2} = \rL^2(\langle \cN \ovt R, e_N \rangle) \;.$$

Note that $\langle P \ovt R, e_{\mcL Z} \rangle \subset \langle \cN_\cF \ovt R, e_{N_\cF} \rangle$ for all $\cF \subset \Gamma$. Denote by $E_\cF$ the trace preserving conditional expectation of $\cN_1$ onto $\langle \cN_\cF \ovt R, e_{N_\cF} \rangle$  Choose $\eps > 0$. Take a large enough finite subset $\cF \subset \Gamma$ such that
$$ \|T - E_\cF(T)\|_2 < \eps \; .$$
Since $T = \rho_g \si_g(T) \rho_g^*$, we get that $\|T - \rho_g \si_g(E_\cF(T)) \rho_g^*\|_2 < \eps$. As $\rho_g \in P \ovt R$, it follows that $T$ lies at distance at most $\eps$ from $\langle \cN_{\cF g^{-1}} \ovt R, e_{N_{\cF g^{-1}}} \rangle$. Since $T$ also lies at distance at most $\eps$ from $\langle \cN_\cF \ovt R, e_{N_\cF} \rangle$, we conclude that $T$ lies at distance at most $2 \eps$ from
$\langle \cN_{\cF \cap  \cF g^{-1}} \ovt R, e_{N_{\cF \cap  \cF g^{-1}}} \rangle$ for all $g \in \Gamma$. We can choose $g$ such that $\cF \cap \cF g^{-1} = \emptyset$ and conclude that $T$ lies at distance at most $2 \eps$ from $P_1$. Since $\eps > 0$ is arbitrary, it follows that $T \in P_1$.

So we can view $T$ as the orthogonal projection of $\rL^2(P \ovt R)$ onto a right $\mcL Z$ submodule of dimension $\tau(p)$. Since $p T = T$, the image of $T$ is contained in $p \rL^2(P \ovt R)$. Take projections $q_n \in \mcL Z$ with $\sum_n \tau(q_n) = \tau(p)$ and a right $\mcL Z$-linear isometry
$$\Theta : \bigoplus_{n \in \N} q_n \rL^2(\mcL Z) \recht p \rL^2(P \ovt R)$$
onto the image of $T$. Denote by $w_n \in p \rL^2(P \ovt R)$ the image under $\Theta$ of $q_n$ sitting in position $n$. Note that
$$w_n = p w_n \;\; , \;\; E_{\mcL Z}(w_n^* w_m) = \delta_{n,m} q_n \quad\text{and}\quad T = \sum_{n \in \N} w_n e_{\mcL Z} w_n^*$$
where in the last formula we view $T$ as an element of $P_1$. Identifying $P_1$ as above with the von Neumann subalgebra of $\cN_1$ generated by $P \ovt R$ and $e_N$, it follows that
$$T = \sum_{n \in \N} w_n e_N w_n^* \; .$$
Since $T = w e_N w^*$, we get that $p e_N = \sum_n w^* w_n e_N w_n^* w$. Denote $x := e_N w_n^* w$. It follows that $|x|^2 = w^* w_n e_N w_n^* w = w^* w_n e_N w_n^* w e_N = |x|^2 e_N$. Then also $|x| = |x| e_N$ and hence $e_N w_n^* w = e_N w_n^* w e_N$. So $w^* w_n \in \rL^2(\cN)$ preserves $\rL^2(N)$. We conclude that $w^* w_n \in \rL^2(N)$ for all $n$. It follows that $w_n^* w_m \in \rL^1(N)$ for all $n,m$. Since also $w_n^* w_m \in \rL^1(P \ovt R)$, we have $w_n^* w_m \in \rL^1(\mcL Z)$. But then,
$$\delta_{n,m} q_n = E_{\mcL Z}(w_n^* w_m) = w_n^* w_m \; .$$
So, the elements $w_n$ are partial isometries in $P \ovt R$ with mutually orthogonal left supports lying under $p$ and with right supports equal to $q_n$. Since $\sum_n \tau(q_n) = \tau(p)$, we conclude that the formula
$$W := \sum_n e_{1,n} \ot w_n$$
defines an element $W \in \B(\ell^2(\N),\C) \ovt P \ovt R$ satisfying $WW^* = p$, $W^* W = q$ where $q \in (\mcL Z)^\infty$ is the projection given by $q := \sum_n e_{nn} \ot q_n$. We also have that $v := W^* w$ belongs to $\B(\C,\ell^2(\N)) \ovt N$ and satisfies $v^* v = p$, $v v^* = q$.

Recall that $T = \rho_g \si_g(T) \rho_g^*$. Let $\gamma : \Gamma \recht \cU(q (\mcL Z)^\infty q)$ be the unique group homomorphism satisfying
$$\Theta(\gamma_g \xi) = \rho_g U_g \Theta(\xi) \quad\text{for all}\;\; g \in \Gamma , \xi \in \bigoplus_n q_n \rL^2(\mcL Z) \; $$
where $U_g$ is the unitary on $\rL^2(P \ovt R)$ implementing the action $\si_g$ on $P_1$.
By construction $W \gamma_g = \rho_g \si_g(W)$ for all $g \in \Gamma$. Since $\om_g = w^* \, \rho_g \, \si_g(w)$, we conclude that $\om_g = v^* \, \gamma_g \, \si_g(v)$.

We finally prove that $\gamma$ is unique up to unitary conjugacy. Assume that we also have a projection $q_1 \in (\mcL Z)^\infty$, a generalized 1-cocycle $\vphi_g$ for the action $\Gamma \actson (\mcL Z)^\infty$ with support $q_1$ and a partial isometry $v_1 \in \B(\C,\ell^2(\N)) \ovt N$ with $v_1^* v_1 = p$ and $v_1 v_1^* = q_1$, such that $\om_g = v_1^* \, \vphi_g \, \si_g(v_1)$. Then the element $v_1 v^* \in q_1 N^\infty q$ satisfies $v_1 v^* = \vphi_g \si_g(v_1 v^*) \gamma_g^*$ for all $g \in \Gamma$. Denote by $\tilde{E}_{\cF}$ the unique trace preserving conditional expectation $\tilde{E}_{\cF}: N \rightarrow N_\cF$. Choose $\epsilon > 0$ and denote by $\| \cdot \|_2$ the 2-norm on $N^\infty$ given by the semi-finite trace $\Tr \ot \tau$. Take a large enough finite subset $\mathcal{F} \subset \Gamma$ such that
$$ \| v_1 v^* - (\id_{\B(\ell^2(\N))} \ot \tilde{E}_{\cF})(v_1 v^*) \|_2 < \epsilon \; .$$
Since $v_1 v^* = \vphi_g \si_g(v_1 v^*) \gamma_g^*$, we also get that
$$\| v_1 v^* - \vphi_g (\id_{\B(\ell^2(\N))} \ot \si_g \circ \tilde{E}_{\cF})(v_1 v^*) \gamma_g^* \|_2 < \epsilon \; .$$
Recall that $\vphi_g, \gamma_g^* \in (\mcL Z)^\infty$. Similarly as above we conclude that $v_1 v^*$ lies at distance at most $2 \epsilon$ from $(\mcL Z)^\infty$. Since $\epsilon > 0$ is arbitrary, it follows that $v_1 v^* \in (\mcL Z)^\infty$, providing the required unitary conjugacy between $\gamma$ and $\vphi$.
\end{proof}

\end{document}